\documentclass{article}
\usepackage[utf8]{inputenc}
\usepackage{amsthm, amssymb, mathtools, hyperref, authblk}
\usepackage[dvipsnames]{xcolor}
\usepackage[inline]{enumitem}

\usepackage[
	backend=biber,
	sorting=nty,
	giveninits=true,
	url=false
]{biblatex}
\AtEveryBibitem{\clearlist{language}}
\renewbibmacro{in:}{}
\numberwithin{equation}{section}
\newtheorem{thm}{Theorem}[section]

\newtheorem{lem}[thm]{Lemma}
\newtheorem{cor}[thm]{Corollary}
\theoremstyle{definition}
\newtheorem{defn}[thm]{Definition}
\theoremstyle{remark}

\newtheorem{question}{Question}

\newcommand{\nat}{\mathbb{N}}
\newcommand{\real}{\mathbb{R}}
\newcommand{\complex}{\mathbb{C}}
\newcommand{\sphere}{\mathbb{S}}

\newcommand{\N}{\mathcal{N}}

\newcommand{\iu}{\mathrm{i}}
\newcommand{\eps}{\varepsilon}
\newcommand{\Lebesgue}{\mathcal{L}}

\newcommand{\dif}{\mathrm{d}}

\DeclareMathOperator{\vspan}{span}
\DeclareMathOperator{\Dom}{Dom}

\DeclarePairedDelimiter{\abs}{\lvert}{\rvert}
\DeclarePairedDelimiter{\norm}{\lVert}{\rVert}
\DeclarePairedDelimiter{\parens}{(}{)}
\DeclarePairedDelimiter{\set}{\{}{\}}
\DeclarePairedDelimiter{\brackets}{\lbrack}{\rbrack}

\DeclarePairedDelimiter{\angles}{\langle}{\rangle}
\DeclarePairedDelimiter{\cci}{\lbrack}{\rbrack}
\DeclarePairedDelimiter{\coi}{\lbrack}{\lbrack}
\DeclarePairedDelimiter{\oci}{\rbrack}{\rbrack}
\DeclarePairedDelimiter{\ooi}{\rbrack}{\lbrack}

\title{Sign-changing multi-peak standing waves of the NLSE with a point interaction}
\begin{document}

\author{Gustavo de Paula Ramos\thanks{gpramos@icmc.usp.br}}
\affil{Universidade de São Paulo, Instituto de Ciências Matemáticas e de Computação, Avenida Trabalhador São-Carlense, 400, 13566-590 São Carlos SP, Brazil}

\maketitle
\begin{abstract}
\noindent
Consider the following semilinear problem with a point interaction in $\real^N$:
\[- \Delta_\alpha u + \omega u = u |u|^{p - 2},\]
where $N \in \{2, 3\}$; $\omega > 0$;
$- \Delta_\alpha$ denotes the Hamiltonian of point interaction with inverse $s$-wave scattering length
$- (4 \pi \alpha)^{- 1}$ and we want to solve for $u \colon \mathbb{R}^N \to \mathbb{R}$. By means of Lyapunov--Schmidt reduction, we prove that this problem has sign-changing multi-peak solutions when either
\begin{enumerate*}
\item
$N = 2$, $\alpha \in \real$,
$p_* < p \leq 3$ and $\omega$ is sufficiently large or
\item
$N = 3$,
$0 < \alpha < \infty$, $p_* < p < 3$ and $\omega$ is sufficiently small,
\end{enumerate*}
where
$2.45 < p_* := \frac{9 + \sqrt{113}}{8} < 2.46$.
\end{abstract}

\tableofcontents

\section{Introduction}

\subsection{Context and motivation}

The presence of punctual impurities is often modeled by means of $\delta$ potentials. For instance, the following formal equation on
$\real$ was proposed in \cite{sakaguchiSingularSolitons2020, shamrizSingularMeanfieldStates2020} to explain the existence of singular solitons of the defocusing septimal nonlinear Schrödinger equation (NLSE):
\[
-
\frac{1}{2} u''
+
\alpha \delta_0 u
+
\omega u
+
u^7
=
0,
\]
where
$\alpha < 0$, $\omega > 0$ and $\delta_0$ denotes the Dirac delta.

In this paper, we consider the model for this kind of interaction in dimensions 2 and 3 furnished by the \emph{Hamiltonian of point interaction}
\[- \Delta_\alpha \colon \Dom \parens{- \Delta_\alpha} \to L^2\]
as defined in Section \ref{sect:Delta_alpha}, where $\Dom \parens{- \Delta_\alpha}$ is a certain subspace of $L^2$. Briefly,
$- \Delta_\alpha$ is defined as an $L^2$-self-adjoint differential operator used to model zero-range interactions with the goal of, at least when $\alpha \neq 0$, mimicking the action of
$- \Delta + \frac{1}{\alpha} \delta_0$ on $C_c^\infty$.

We are concerned with the following semilinear equation with a point interaction in $\real^N$:
\begin{equation}
\label{eqn:delta-NLS}
- \Delta_\alpha u + \omega u = f \parens{u},
\end{equation}
where $N \in \set{2, 3}$; $\omega > 0$; $\alpha \in \real$;
$f \parens{t} := t \abs{t}^{p - 2}$ for every $t \in \real$ ($p > 2$) and we want to solve for $u \colon \real^N \to \real$. Problem \eqref{eqn:delta-NLS} is obtained when looking for standing wave solutions with an \emph{a priori} known energy $\omega$ to the following focusing NLSE with a point interaction in $\real^N$
($\delta$-NLSE):
\begin{equation}
\label{eqn:the-delta-NLSE}
- \Delta_\alpha \psi = \iu \psi_t + f \parens{\psi}.
\end{equation}

Even though the definition of $- \Delta_\alpha$ remounts to Berezin \& Faddeev's \cite{berezinRemarkSchrodingerEquation1961}, the study of nonlinear problems with a point interaction in higher dimensions has only gained traction a few years ago. Let us do a brief review of recent developments. The well-posedness in the operator domain of the $\delta$-NLSE
\[- \Delta_\alpha \psi = \iu \psi_t \pm f \parens{\psi}\]
was established by Cacciapuoti, Finco \& Noja in \cite{cacciapuotiWellPosednessNonlinear2021}.
The existence and orbital stability of minimizers of the action functional associated with \eqref{eqn:delta-NLS} in dimension 2 were addressed by Fukaya, Georgiev \& Ikeda in \cite{fukayaStabilityInstabilityStanding2022}. In \cite{adamiGroundStatesPlanar2022, adamiExistenceStructureRobustness2022}, Adami, Boni, Carlone \& Tentarelli proved the existence of ground states for \eqref{eqn:the-delta-NLSE} and obtained some of their qualitative properties in dimensions 2 and 3. Even nonlinear problems with a point interaction and a nonlocal nonlinearity have started being considered, such as the Hartree equation (see \cite{georgievStandingWavesGlobal2024, michelangeliSingularHartreeEquation2021}) and the Kirchhoff equation or the Schrödinger--Poisson system (see \cite{depaularamosMinimizersConstrainedFunctionals2024}).

We need to recall a well-known result about multi-peak standing waves of the NLSE in the semiclassical regime to explain the motivation of this paper. Consider the following semilinear problem in $\real^N$:
\begin{equation}
\label{eqn:NLS-with-V}
- \eps^2 \Delta u + V \parens{x} u = f \parens{u},
\end{equation}
where $\eps > 0$ and
$V \colon \real^N \to \ooi{0, \infty}$ has a local maximum point $x_0 \in \real^3$. In \cite{kangInteractingBumpsSemiclassical2000}, Wei \& Kang proved that given a positive integer $K$,
\begin{enumerate*}
\item
there exists
$\eps_K > 0$ such that if
$0 < \eps < \eps_K$, then \eqref{eqn:NLS-with-V} has a solution with $K$ peaks and
\item
the peaks of these solutions collapse around $x_0$ as $\eps \to 0^+$. See also \cite{pistoiaMultipeakSolutionsClass2002,daprileNumberSignchangingSolutions2007,ruizClusterSolutionsSchrodingerPoissonSlater2011, ianniNonradialSignchangingSolutions2015} for related results.
\end{enumerate*}

These solutions arise due to the balance of two opposing effects on the peaks: the attractive effect of the nonlinearity and the repulsive effect of the potential around $x_0$. At least when $N = 3$ and $0 \leq \alpha \leq \infty$, the operator $-\Delta_\alpha$ similarly models a repulsive point interaction (see Section \ref{sect:functional-framework}). As such, one may ask whether a similar balance could be achieved, hence the question that follows.

\begin{question}
\label{question}
Does \eqref{eqn:delta-NLS} admit solutions with multiple peaks in limit situations of the energy $\omega$?
\end{question}

In this paper, we answer this question positively for sign-changing solutions with an arbitrary number of peaks. Rescalings of Problem \eqref{eqn:delta-NLS} behave quite differently in dimensions 2 and 3 (see Appendix \ref{sect:equivalent-problem}), so we will actually consider two limit situations of
$\omega$ according to the considered dimension.

To finish, we comment on the organization of the rest of the introduction.
\begin{itemize}
\item
We fix the notation used throughout the paper in Section \ref{notation}.
\item
Section \ref{sect:Delta_alpha} recalls the definition of the Hamiltonian of point interaction
$- \Delta_\alpha$.
\item
We introduce the appropriate Hilbert space on which to look for weak solutions to the considered problem in Section \ref{sect:functional-framework}.
\item
Finally, Section \ref{sect:main-result} contains a precise statement of our main result.
\end{itemize}

\subsection{Notation and conventions}
\label{notation}

\begin{itemize}
\item
We exclusively consider vector spaces over $\real$.
\item
Square brackets are used to enclose the argument of (multi-)linear mappings.
\item
Given $r \in \cci{1, \infty}$, we let $r'$ denote its Hölder conjugate exponent.
\item
Except mentioned otherwise, we suppose that \begin{itemize}
\item
$N \in \set{2, 3}$;
\item
$\alpha \in \real$;
\item
$\eta \in \ooi{0, \infty}$.
\end{itemize}
\item
As for $\real^N$:
\begin{itemize}
\item
$B \parens{x, r}$ denotes the open ball centered at $x$ with radius $r$;
\item
we let $\sphere^{N - 1} = \set{x \in \real^N : \abs{x} = 1}$;
\item
$\Lebesgue$ denotes the Lebesgue measure;
\item
$\parens{x, y} \mapsto x \cdot y$ denotes the Euclidean inner product.
\end{itemize}
\item
Given a Hilbert space $H$,
\begin{itemize}
\item
$\parens{h_1, h_2} \mapsto \angles{h_1 \mid h_2}_H$ denotes the inner product of $H$ and
\item
if $f \colon H \to \real$ is differentiable, then $\nabla f \parens{x}$ denotes the gradient of $f$ at $x \in H$.
\end{itemize}
\item
Suppose that $X$, $Y$ are sets and consider mappings
$f \colon X \times Y \to \coi{0,\infty}$, $g \colon Y \to \coi{0, \infty}$. Given an element $x \in X$, we write
\[
\text{``}
f \parens{x, y} \lesssim g \parens{y} \quad \text{for every} \quad y \in Y
\text{''}
\]
to mean that there exists $C_x \in \ooi{0, \infty}$ such that $f \parens{x, y} \leq C_x g \parens{y}$ for every $y \in Y$.
\end{itemize}

\subsection{The Hamiltonian of point interaction}
\label{sect:Delta_alpha}

Consider the operator
$
- \Delta|_{
	C^\infty_0
	\parens{
		\real^N
		\setminus
		\set{0}
	}
}
$.
In dimension $N \geq 4$, this operator is essentially $L^2$-self-adjoint, so its only $L^2$-self-adjoint extension is
$
- \Delta|_{H^2 \parens{\real^N}}
$
(see \cite[p. 2]{albeverioSolvableModelsQuantum1988}).
In the lower-dimensional case
$N \in \set{2, 3}$, this operator admits a family of nontrivial $L^2$-self-adjoint extensions, which we denote by
$\set{- \Delta_\alpha}_{\alpha \in \oci{- \infty, \infty}}$ (see \cite[Chapters I.1 and I.5]{albeverioSolvableModelsQuantum1988}). We let $\alpha = \infty$ denote the Friedrichs extension
$
- \Delta_\infty
=
- \Delta|_{H^2 \parens{\real^N}}
$
as usual and we henceforth focus on the case
$\alpha \in \real$.

In order to describe the domain of
$- \Delta_\alpha$, we need to introduce a family of Green's functions. Given
$\lambda > 0$, we let
$
G_\lambda
\colon
\real^N \setminus \set{0}
\to
\ooi{0, \infty}
$
be given by
\[
G_\lambda \parens{x}
=
\frac{K_{\frac{N - 2}{2}} \parens*{\sqrt{\lambda} \abs{x}}}
	{\parens{2 \pi}^{\frac{N}{2}} \abs{x}^{\frac{N - 2}{2}}},
\]
so that
\[
\parens{- \Delta + \lambda} G_\lambda
=
\delta_0
\]
in the sense of distributions, where $K_\nu$ denotes a modified Bessel function of the second kind as defined in \cite[Section 9.6]{abramowitzHandbookMathematicalFunctions1972}. In the case $N = 3$, $G_\lambda$ is explicitly given by
\[
G_\lambda \parens{x}
=
\frac{
	e^{- \sqrt{\lambda} \abs{x}}
}{4 \pi \abs{x}}
\]
for every
$
x \in \real^3 \setminus \set{0}
$.
To simplify the notation, we will interpret the absence of the subscript $\lambda$ as the case
$\lambda = 1$, that is, we set $G = G_1$. To finish our discussion about Green's functions, we remark that
\[
G \in L^r
\quad \text{when} \quad
\begin{cases}
1 \leq r < \infty,
&\text{if} ~ N = 2;
\\
1 \leq r < 3,
&\text{if} ~ N = 3.
\end{cases}
\]

We can already proceed to a more detailed description of
$- \Delta_\alpha$. Its domain is given by
\begin{multline*}
\Dom \parens{- \Delta_\alpha}
=
\left\{
	\phi_\lambda + q G_\lambda :
	\phi_\lambda \in H^2, \quad
	q \in \real, \quad
	\lambda > 0
\right.
\\
\left.
	\text{and} \quad
	\beta_\alpha
	\parens{\lambda}
	q
	=
	\phi_\lambda \parens{0}
\right\},
\end{multline*}
where
\[
\beta_\alpha \parens{\lambda}
:=
\begin{cases}
{\displaystyle
	\alpha + \frac{\gamma - \log 2}{2 \pi} + \frac{1}{2 \pi} \log \sqrt{\lambda}
},
&\text{if} ~ N = 2;
\\
\\
{\displaystyle \alpha + \frac{\sqrt{\lambda}}{4 \pi}},
&\text{if} ~ N = 3
\end{cases}
\]
($\gamma$ denoting the Euler--Mascheroni constant). The operator $- \Delta_\alpha$ acts as
\[
- \Delta_\alpha u = - \Delta \phi_\lambda - \lambda q G_\lambda
\quad \text{for every} \quad
u = \phi_\lambda + q G_\lambda \in \Dom \parens{- \Delta_\alpha},
\]
or, equivalently,
\begin{equation}
\label{eqn:trick}
\parens{- \Delta_\alpha + \lambda} u
=
\parens{- \Delta + \lambda} \phi_\lambda
\quad \text{for every} \quad
u = \phi_\lambda + q G_\lambda \in \Dom \parens{- \Delta_\alpha}.
\end{equation}

\subsection{The energy space $H^1_\eta$}
\label{sect:functional-framework}

In view of \cite[Theorems 1.1.4 \& 5.4]{albeverioSolvableModelsQuantum1988}, the spectrum of $- \Delta_\alpha$ is described as follows:
\[
\begin{array}{l l}
\sigma \parens{- \Delta_\alpha}
=
\set{- 4 e^{- 2 \gamma - 4 \pi \alpha}}
\cup
\coi{0, \infty},
&\text{if} \quad N = 2;
\\ \\
\sigma \parens{- \Delta_\alpha}
=
\set{- \parens{4 \pi \alpha}^2}
\cup
\coi{0, \infty},
&\text{if} \quad
N = 3
\quad \text{and} \quad
\alpha < 0;
\\ \\
\sigma \parens{- \Delta_\alpha}
=
\coi{0, \infty}
&\text{if} \quad
N = 3
\quad \text{and} \quad
\alpha \geq 0.
\end{array}
\]
In particular, if
$
\omega > \inf \sigma \parens{- \Delta_\alpha}
$,
then the form
\[
\Dom \parens{- \Delta_\alpha} \ni u
\mapsto
\angles*{
	\parens{- \Delta_\alpha + \omega} u 
	~\middle|~
	u
}_{L^2}
\]
is positive-definite. By proceeding as such, we obtain a family of positive-definite forms that induce equivalent norms (see \cite[Section 1.15]{galloneSelfAdjointExtensionSchemes2023}).

In view of the change of variable in Appendix \ref{sect:equivalent-problem}, we will be mostly interested in the energy space naturally associated with $- \Delta_\eta + 1$ with
$\eta > 0$, which induces a positive-definite form due to the previous discussion. More precisely, we define the \emph{energy space} $H^1_\eta$ as the Hilbert space
\[
H^1_\eta
=
\parens*{
	\Dom \brackets{- \Delta_\eta + 1},
	\angles{\cdot \mid \cdot}_{H^1_\eta}
},
\]
where the \emph{form domain}
$\Dom \brackets{- \Delta_\eta + 1}$
is explicitly given by
\[
\Dom \brackets{- \Delta_\eta + 1}
:=
\set*{
	\phi + q G:
	\parens{\phi, q}
	\in
	H^1 \times \real
}
\subset
L^2
\]
and the inner product
$\angles{\cdot \mid \cdot}_{H^1_\eta}$
is given by
\[
\angles{u_1 \mid u_2}_{H^1_\eta}
:=
\angles{\phi_1 \mid \phi_2}_{H^1}
+
\beta_\eta \parens{1} q_1 q_2
\]
for every
$
u_j
=
\phi_j + q_j G
\in
\Dom \brackets{- \Delta_\eta}
$.
It is clear that
$H^1 \hookrightarrow H^1_\eta$ and
$\norm{\phi}_{H^1_\eta} = \norm{\phi}_{H^1}$ for every $\phi \in H^1$. Furthermore, it follows from the Sobolev embeddings of $H^1$ that
\begin{equation}
\label{eqn:Sobolev-embedding}
H^1_\eta \hookrightarrow L^r
\quad \text{when} \quad
\begin{cases}
2 \leq r < \infty,
&\text{if} ~ N = 2;
\\
2 \leq r < 3,
&\text{if} ~ N = 3.
\end{cases}
\end{equation}

\subsection{Main result}
\label{sect:main-result}

In Appendix \ref{sect:equivalent-problem}, we show that, up to rescaling,
\[- \Delta_\alpha u + \omega u = f \parens{u}\]
is equivalent to
\[
- \Delta_{\alpha_\omega} u + u = f \parens{u},
\]
where
\[
\alpha_\omega
:=
\begin{cases}
{\displaystyle \alpha + \frac{1}{2 \pi} \log \sqrt{\omega}},
&\text{if} ~ N = 2;
\\
\\
{\displaystyle \frac{\alpha}{\sqrt{\omega}}},
&\text{if} ~ N = 3.
\end{cases}
\]
One may expect $- \Delta_\eta$ to behave increasingly similarly as
$- \Delta_\infty = -\Delta|_{H^2}$
as one successively considers larger positive values for $\eta$ and it is clear that
\[
\alpha_\omega \to \infty
\quad \text{as} \quad
\begin{cases}
\omega \to \infty,
&\text{if} ~ N = 2;
\\
\omega \to 0^+,
&\text{if} ~ N = 3 ~ \text{and} ~ 0 < \alpha < \infty.
\end{cases}
\]
As such, we will suppose that either
\begin{itemize}
\item
$N = 2$ and $\omega$ is sufficiently large or
\item
$N = 3$, $0 < \alpha < \infty$ and $\omega$ is sufficiently small,
\end{itemize}
in which case it suffices to consider the problem
\begin{equation}
\label{eqn:delta-NLS-eta}
-\Delta_\eta u + u = f \parens{u}
\end{equation}
for a sufficiently large positive $\eta$. In this context, we understand a \emph{weak solution in $H^1_\eta$} to \eqref{eqn:delta-NLS-eta} to be a critical point of its associated \emph{action functional}
$S_\eta \colon H^1_\eta \to \real$ defined as
\begin{equation}
\label{eqn:action-functional}
S_\eta \parens{u}
=
\frac{1}{2} \norm{u}_{H^1_\eta}^2
-
\frac{1}{p} \norm{u}_{L^p}^p.
\end{equation}

Let us recall a few facts about the following semilinear problem:
\begin{equation}
\label{eqn:Kwong}
\begin{cases}
- \Delta u + u = u \abs{u}^{p - 2}
&\text{in} ~ \real^N;
\\
u \parens{x} \to 0
&\text{as} ~ \abs{x} \to \infty.
\end{cases}
\end{equation}
It is classical that \eqref{eqn:Kwong} has a unique positive radial solution of class $C^2$ (see \cite{kwongUniquenessPositiveSolutions1989}), which we denote by
$\Phi \colon \real^N \to \ooi{0, \infty}$. Furthermore,
$\Phi \in H^1$ and
\begin{equation}
\label{eqn:Phi-limit}
\abs{x}^{\frac {N - 1}{2}} e^{\abs{x}} \Phi \parens{x}
\xrightarrow[\abs{x} \to \infty]{}
\theta_\Phi
:=
\int e^{x_1} \Phi \parens{x}^{p - 1} \dif \Lebesgue \parens{x}
\in
\ooi{0, \infty}
\end{equation}
(see \cite{gidasSymmetryPositiveSolutions1981}). It will be important to consider translations of $\Phi$, so we set
\[\Phi_y = \Phi \parens{\cdot - y}\]
for every $y \in \real^N$. To finish, we recall that the critical points of the action functional $S_\infty \colon H^1 \to \real$ are weak solutions to \eqref{eqn:Kwong}, where
\[
S_\infty \parens{\phi}
:=
\frac{1}{2} \norm{\phi}_{H^1}^2 - \frac{1}{p} \norm{\phi}_{L^p}^p.
\]

We proceed to the statement of our main result.
\begin{thm}
\label{thm}
Consider an integer $K \geq 2$ and suppose that either
\begin{itemize}
\item
$N = 2$, $\alpha \in \real$ and
$p_* < p \leq 3$;
\item
$N = 3$, $0 < \alpha < \infty$ and
$p_* < p < 3$,
\end{itemize}
where
\[
2.45 < p_* := \frac{9 + \sqrt{113}}{8} < 2.46.
\]
Fix $\delta_1, \ldots, \delta_K \in \set{-1, 1}$ such that
\begin{equation}
\label{eqn:sign-condition}
\delta_K \delta_1 + \delta_1 \delta_2 + \ldots + \delta_{K - 1} \delta_K < 0.
\end{equation}
Then there exist $\eta_K > 1$ and sets
$
\set{r_\eta}_{\eta \geq \eta_K}
\subset
\ooi{0, \infty}
$,
\[
\set*{
	u_\eta \in L^2:
	u_\eta \in H^1_\eta
	\quad \text{and} \quad
	\nabla S_\eta 	
	\parens{u_\eta}
	=
	0
}_{\eta \geq \eta_K}
\]
such that
\[
\norm*{u_\eta - \sum_{1 \leq k \leq K} \delta_k \Phi_{\zeta_{r_\eta}^k}}_{H^1_\eta}
\xrightarrow[\eta \to \infty]{}
0
\quad \text{and} \quad
\frac{r_\eta}{\log \eta}
\xrightarrow[\eta \to \infty]{}
1,
\]
where
\[
\zeta_r^k
:=
\begin{cases}
{\displaystyle
	\frac{3 r}{2 \sin \parens*{\frac{\pi}{K}}}
	\parens{e^{\frac{2 \pi \iu k}{K}} - 1}
	-
	r
	\in \complex \simeq \real^2,
}
&\text{if} ~ N = 2;
\\
\\
{\displaystyle
	\parens*{
		\frac{3 r}{2 \sin \parens*{\frac{\pi}{K}}}
		\parens{e^{\frac{2 \pi \iu k}{K}} - 1}
		-
		r
		,
		0
	}
	\in
	\complex \times \real \simeq \real^3,
}
&\text{if} ~ N = 3
\end{cases}
\]
for every $k \in \set{1, \ldots, K}$ and
$r > 0$.\footnote{More intuitively, $\zeta^1_r, \ldots, \zeta^K_r$ denote the vertices of a (possibly degenerate) regular polygon with $K$ sides and side length
$3 r$; $\zeta^K_r$ denotes the vertex which is closest to origin and $\abs{\zeta^K_r} = r$.}
\end{thm}

A few remarks are in order.
\begin{itemize}
\item
Due to the sign condition \eqref{eqn:sign-condition}, the theorem exclusively provides sign-changing solutions.
\item
The peaks of the solutions furnished by the theorem collapse around the origin as
$\omega \to \infty$ in the case $N = 2$. Indeed, if
$u \in \Dom \parens{- \Delta_\alpha}$ satisfies
$- \Delta_\alpha u + \omega u = f \parens{u}$, then
\[
\widetilde{u}_\omega \parens{x}
:=
\omega^{- \frac{1}{p - 2}} u \parens{\omega^{- \frac{1}{2}} x}
\]
satisfies
$
- \Delta_{\alpha_\omega} \widetilde{u}_\omega + \widetilde{u}_\omega
=
f \parens{\widetilde{u}_\omega}
$
(see Appendix \ref{sect:equivalent-problem}) and the distance from the peaks to the origin is proportional to
$\log \parens{\alpha + \frac{1}{2 \pi} \log \sqrt{\omega}}$.
\end{itemize}

To the best of our knowledge, Theorem \ref{thm} is the first result about sign-changing or multi-peak solutions to a semilinear problem involving a point interaction. We are also unaware of any other paper that used Lyapunov--Schmidt reduction in a similar context.

Let us comment about the strategy of its proof. In view of the definition of
$\norm{\cdot}_{H^1_\eta}$, the point interaction only affects functions in $H^1_\eta \setminus H^1$. As such, we have to consider a building block for pseudo-critical points of $S_\eta$ with a nonzero charge
$q \in \real$. More precisely, we adopt the \emph{building block} obtained by associating each $y \in \real^N$ with
\[
\Psi_{\eta, y}
:=
\Phi_y + \frac{\Phi \parens{y}}{\beta_\eta \parens{1}} G
\in
\Dom \parens{- \Delta_\eta}
\]
and we adopt the following associated \emph{pseudo-critical point} of $S_\eta$:
\[
W_{\eta, r}
:=
\sum_{1 \leq k \leq K}
\parens{\delta_k \Psi_{\eta, \zeta_r^k}}
\in
\Dom \parens{- \Delta_\eta}.
\]
Similarly as in \cite{kangInteractingBumpsSemiclassical2000}, there are two effects on sums of these building blocks: the effect of the point interaction and the attraction (resp., repulsion) between the peaks of the same sign (resp., of different sign) due to the power nonlinearity. The point effect over
$\Psi_{\eta, y}$ is proportional to
$\frac{\Phi \parens{y}^2}{\beta_\eta \parens{1}}$, while the effect of the nonlinearity on
$\Phi_{y_1} + \Phi_{y_2}$ is proportional to
$\Phi \parens{y_1 - y_2}$. This observation inspired the consideration of peaks disposed along a regular polygon whose side length is greater than the Euclidean norm of its vertex that is closest to the origin.

We proceed to a discussion about two technical hypotheses of the theorem. 
\begin{itemize}
\item
Our result is only valid for $p > p_*$ instead of $p > 2$ because we could not prove that the considered pseudo-critical point is a good approximation for a solution when
$2 < p \leq p_*$. Indeed, Lemma \ref{lem:pseudo-critical} shows that
\[
\norm*{\nabla S_\eta \parens{W_{\eta, r}}}_{H^1_\eta}
\lesssim
\frac
	{r^{\frac{\parens{3 - p} N + p - 1}{2 p '}}}
	{
		\eta^{
			\min
			\parens*{
				\frac{3}{p '},
				2 \parens{p - 2} + \frac{1}{p '}
			}
		}
	}
\]
and
\[
2 \parens*{2 \parens{p - 2} + \frac{1}{p '}} > 3
\quad \text{if, and only if,} \quad
p > p_*.
\]
As such, the error bound in Corollary \ref{cor:expansion} is greater than the effect of the point interaction when $2 < p \leq p_*$.
\item
The theorem only holds under the hypothesis $p \leq 3$ in dimension 2 even though $S_\eta$ is well-defined for any $p \geq 2$. We need this restriction to obtain an upper bound for the variation of $S_\eta''$ in a neighborhood of $W_{\eta, r}$. This is done in Lemma \ref{lem:estimate-second-derivative}, where we estimate a term of the form
\[
\norm*{
	\abs{W_{\eta, r} + v}^{p - 2}
	-
	\abs{W_{\eta, r}}^{p - 2}
}_{L^{p'}}
\]
exclusively in function of $v \in H^1_\eta$. The function $W_{\eta, r}$ is not bounded, so we can only obtain such an estimate when $0 < p - 2 \leq 1$, in which case
$\coi{0, \infty} \ni t \mapsto t^{p - 2}$ is sub-additive (see Item \eqref{elementary:1} in Lemma \ref{lem:elementary}), hence the restriction.
\end{itemize}

In bifurcation theory, Lyapunov--Schmidt reduction is usually used to prove the existence of bifurcation points. Let us explain how to interpret Theorem \ref{thm} from a point of view derived from this context. Given $\eta \in \oci{0, \infty}$, consider the mapping
\[
H^1_\eta \ni u
\mapsto
F \parens{\eta, u} := \nabla S_\eta \parens{u}
\in H^1_\eta,
\]
so that $F$ denotes a mapping defined on
\[
\mathcal{G}
:=
\set*{
	\parens{\eta, u} :
	\eta \in \oci{0, \infty} \quad \text{and} \quad u \in H^1_\eta
}.
\]
The set of the usual weak solutions to the NLSE
$- \Delta \phi + \phi = \phi \abs{\phi}^{p - 2}$, that is,
\[Z := \set*{\phi \in H^1 : \nabla S_\infty \parens{\phi} = 0},\]
becomes the \emph{set of trivial solutions} and Theorem \ref{thm} ensures the existence of nontrivial nodal solutions to the problem
\[F \parens{\eta, u} = 0; \quad \parens{\eta, u} \in \mathcal{G}.\]
We remark that the presently considered context does not fit the usual situation in bifurcation theory because $\mathcal{G}$ is not locally trivial in the sense of Hilbert space bundles due to the presence of $H^1 = H^1_\infty$, which is not naturally isomorphic to $H^1_\eta$ with
$0 < \eta < \infty$.

\subsection*{Acknowledgment}

\sloppy
This study was financed, in part, by the São Paulo Research Foundation (FAPESP), Brasil. Process Number \#2024/20593-0.

\section{Existence of sign-changing multi-peak solutions}
\label{sect:multi-peak}

Our present goal is to prove Theorem \ref{thm}. As such, we suppose that its hypotheses are satisfied throughout this section.

\subsection{Pseudo-critical points and admissible distances}
\label{sect:pseudo}

We are interested in solutions with $K$ peaks, so we associate each
$r > 0$ with the following \emph{pseudo-critical point} of $S_\eta$:
\[
W_{\eta, r}
:=
\sum_{1 \leq k \leq K}
\parens{\delta_k \Psi_{\eta, \zeta_r^k}}
\in
\Dom \parens{- \Delta_\eta},
\]
where we recall that
\[
\Psi_{\eta, y}
:=
\Phi_y + \frac{\Phi \parens{y}}{\beta_\eta \parens{1}} G
\in
\Dom \parens{- \Delta_\eta}
\]
for every $y \in \real^N$. Actually, we will only consider values of $r$ in a specific bounded interval. We need a few preliminaries to define this interval. More precisely, let
\begin{equation}
\label{eqn:chi}
\chi
=
-
\frac{4 - p}{2 p}
\times
\begin{cases}
\delta_1 \delta_2,
&\text{if} ~ K = 2;
\\
\delta_K \delta_1
+
\delta_1 \delta_2
+
\ldots
+
\delta_{K - 1} \delta_K,
&\text{if} ~ K \geq 3,
\end{cases}
\end{equation}
which is positive due to the sign condition \eqref{eqn:sign-condition}, and fix $c > 4 \chi$. With these quantities in mind, we define the following interval of \emph{admissible distances}:
\[
R_\eta
:=
\set*{
	r > 0:
	\frac{\eta}{\log \eta}
	<
	\Phi \parens{r}^2
	\parens*{\int \Phi_{3 r} \Phi^{p - 1} \dif \Lebesgue}^{- 1}
	<
	c \eta
}.
\]
As $\eta \to \infty$, distances in $R_\eta$ become uniformly closer to $\log \eta$ as stated more precisely in the next result.

\begin{lem}
\label{lem:asymptotic}
The limit that follows is satisfied:
\[
\sup_{r \in R_\eta} \abs*{\frac{r}{\log \eta} - 1}
\xrightarrow[\eta \to \infty]{}
0.
\]
\end{lem}
\begin{proof}
Due to Lemmas \ref{lem:error-in-limit} and \ref{lem:interaction},
\begin{equation}
\label{lem:asymptotic:1}
C \parens{r}
:=
\frac
{
	\Phi \parens{r}^2
	\parens*{
		\int
			\Phi_{3 r} \Phi^{p - 1}
	\dif \Lebesgue
	}^{- 1}
}
{
	\theta_\Phi e^r
	\parens*{\frac{3}{r}}^{\frac{N - 1}{2}}
}
\xrightarrow[r \to \infty]{}
1.
\end{equation}
It is clear that
\[
\coi{0, \infty} \ni r
\mapsto
\Phi \parens{r}^2
\parens*{
	\int
		\Phi_{3 r} \Phi^{p - 1}
	\dif \Lebesgue
}^{- 1}
\in \ooi{0, \infty}
\]
is a well-defined continuous function, so it follows from the Squeeze Theorem that
\begin{equation}
\label{lem:asymptotic:2}
\inf_{r \in R_\eta} r
\xrightarrow[\eta \to \infty]{}
\infty.
\end{equation}
Given $r \in R_\eta$, it holds that
\[
\frac{\eta}{\log \eta}
<
\theta_\Phi
e^r
\parens*{\frac{3}{r}}^{\frac{N - 1}{2}}
C \parens{r}
<
c \eta,
\]
so
\[
1 - \frac{\log \log \eta}{\log \eta}
<
\frac
{
	\log \theta_\Phi
	+
	r	
	+
	\frac{N - 1}{2} \log \frac{3}{r}
	+
	\log C \parens{r}
}
{\log \eta}
<
1 + \frac{\log c}{\log \eta}.
\]
In this situation, the result follows from \eqref{lem:asymptotic:1} and \eqref{lem:asymptotic:2}.
\end{proof}

We remark that even though the proof of Lemma \ref{lem:asymptotic} used the limit
\[
\abs*{\Phi \parens{r} - \theta_\Phi \frac{e^{- \abs{r}}}{r^{\frac{N - 1}{2}}}}
\xrightarrow[r \to \infty]{}
0,
\]
the set $R_\eta$ is not defined using this approximation for the behavior of $\Phi$. This is because the error associated with this limit is larger than the effect of the point interaction in dimension 3. More precisely, the effect of the point interation is of order
\[
\frac{\Phi \parens{r}^2}{\eta}
\simeq
\frac{e^{- 2 r}}{\eta r^{N - 1}}
\simeq
\frac{1}{\eta^3 r^{N - 1}}
\]
(see Lemma \ref{lem:expansion}), while the aforementioned error to estimate
$
\int
	\Phi_{3 r} \Phi^{p - 1}
\dif \Lebesgue
$
is of order
\[
\frac{e^{- 3 r}}{r^{\frac{N + 1 - \eps}{2}}}
\simeq
\frac{1}{\eta^3 r^{\frac{N + 1 - \eps}{2}}}
\]
(see Lemma \ref{lem:interaction}).

The function $W_{\eta, r}$ is said to be a pseudo-critical point of $S_\eta$ due to the following result.

\begin{lem}
\label{lem:pseudo-critical}
There exists $\eta_0 > 1$ such that
\[
\norm*{\nabla S_\eta \parens{W_{\eta, r}}}_{H^1_\eta}
\lesssim
\frac
	{r^{\frac{\parens{3 - p} N + p - 1}{2 p '}}}
	{
		\eta^{
			\min
			\parens*{
				\frac{3}{p '},
				2 \parens{p - 2} + \frac{1}{p '}
			}
		}
	}
\]
for every $\eta \geq \eta_0$ and $r \in R_\eta$.
\end{lem}
\begin{proof}
In view of \eqref{eqn:trick},
\begin{equation}
\label{pseudo-critical:0.05}
\angles*{
	\nabla S_\eta \parens{W_{\eta, r}}
	~ \middle| ~
	u
}_{H^1_\eta}
=
\eqref{pseudo-critical:0.1}
-
\eqref{pseudo-critical:0.2},
\end{equation}
where
\begin{equation}
\label{pseudo-critical:0.1}
\sum_{1 \leq k \leq K}
\parens*{
	\delta_k
	\angles{
		-
		\Delta \Phi_{\zeta_r^k}
		+
		\Phi_{\zeta_r^k}
		\mid
		u
	}_{L^2}
}
\end{equation}
and
\begin{equation}
\label{pseudo-critical:0.2}
\int f \parens{W_{\eta, r}} u \dif \Lebesgue.
\end{equation}
On one hand, the identity
$- \Delta \Phi + \Phi = \Phi^{p - 1}$
shows that
\begin{equation}
\label{pseudo-critical:0.3}
\eqref{pseudo-critical:0.1}
=
\int
	\sum_{1 \leq k_1 \leq K}
	\parens{
		\delta_{k_1}
		\Phi_{\zeta_r^{k_1}}^{p - 1}
	}
	u
\dif \Lebesgue.
\end{equation}
On the other hand, the equality
\[
W_{\eta, r}
=
\sum_{1 \leq k_1 \leq K}
\parens*{
	\delta_{k_1}
	\parens*{
		\Phi_{\zeta_r^{k_1}}
		+
		\frac
		{\Phi \parens{\zeta_r^{k_1}}}
		{\beta_\eta \parens{1}}
		G
	}
}
\]
implies
\begin{multline}
\label{pseudo-critical:0.4}
\eqref{pseudo-critical:0.2}
=
\int
	\sum_{1 \leq k_1 \leq K}
	\parens*{
		\delta_{k_1} \Phi_{\zeta_r^{k_1}}
		\abs{W_{\eta, r}}^{p - 2}
	}
	u
\dif \Lebesgue
+
\\
+
\parens*{
	\sum_{1 \leq k_1 \leq K}
	\delta_{k_1}
	\frac
		{\Phi \parens{\zeta_r^{k_1}}}
		{\beta_\eta \parens{1}}
}
\int
	G \abs{W_{\eta, r}}^{p - 2} u
\dif \Lebesgue.
\end{multline}
Due to \eqref{pseudo-critical:0.05}, \eqref{pseudo-critical:0.3} and \eqref{pseudo-critical:0.4}, we obtain
\begin{multline*}
\angles*{
	\nabla S_\eta \parens{W_{\eta, r}}
	~ \middle| ~
	u
}_{H^1_\eta}
=
\int
	\sum_{1 \leq k_1 \leq K}
	\parens*{
		\delta_{k_1}
		\Phi_{\zeta_r^{k_1}}
		\parens*{
			\Phi_{\zeta_r^{k_1}}^{p - 2}
			-
			\abs{W_{\eta, r}}^{p - 2}
		}
	}
	u
\dif \Lebesgue
+
\\
-
\parens*{
	\sum_{1 \leq k_1 \leq K}
	\delta_{k_1}
	\frac
		{\Phi \parens{\zeta_r^{k_1}}}
		{\beta_\eta \parens{1}}
}
\int
	G \abs{W_{\eta, r}}^{p - 2} u
\dif \Lebesgue.
\end{multline*}
By summing and subtracting
\[
\int
\sum_{1 \leq k_1 \leq K}
	\parens*{
		\delta_{k_1} \Phi_{\zeta_r^{k_1}}
		\abs*{
			\sum_{1 \leq k_2 \leq K} 
				\parens{
					\delta_{k_2}
					\Phi_{\zeta_r^{k_2}}
				}
		}^{p - 2}
	}
u
\dif \Lebesgue,
\]
we deduce that
\begin{multline*}
\angles*{\nabla S_\eta \parens{W_{\eta, r}} ~ \middle| ~ u}_{H^1_\eta}
=
\\
=
\int
	\sum_{1 \leq k_1 \leq K}
		\parens*{
			\delta_{k_1} \Phi_{\zeta_r^{k_1}}
			\parens*{
				\Phi_{\zeta_r^{k_1}}^{p - 2}
				-
				\abs*{\sum_{1 \leq k_2 \leq K} \delta_{k_2} \Phi_{\zeta_r^{k_2}}}^{p - 2}
			}
		}
	u
\dif \Lebesgue
+
\\
+
\int
	\sum_{1 \leq k_1 \leq K}
		\parens*{
		\delta_{k_1} \Phi_{\zeta_r^{k_1}}
			\parens*{
				\abs*{\sum_{1 \leq k_2 \leq K} \delta_{k_2} \Phi_{\zeta_r^{k_2}}}^{p - 2}
				-
				\abs{W_{\eta, r}}^{p - 2}
			}
		}
	u
\dif \Lebesgue
+
\\
-
\parens*{
	\sum_{1 \leq k_1 \leq K}
	\delta_{k_1}
	\frac{\Phi \parens{\zeta_r^{k_1}}}{\beta_\eta \parens{1}}
}
\int G \abs{W_{\eta, r}}^{p - 2} u \dif \Lebesgue.
\end{multline*}
As such, we only have to estimate
\begin{equation}
\label{pseudo-critical:1}
\sum_{1 \leq k_1 \leq K}
\norm*{
	\Phi_{\zeta_r^{k_1}}
	\parens*{
		\Phi_{\zeta_r^{k_1}}^{p - 2}
		-
		\abs*{\sum_{1 \leq k_2 \leq K} \delta_{k_2} \Phi_{\zeta_r^{k_2}}}^{p - 2}
	}
}_{L^{p'}},
\end{equation}
\begin{equation}
\label{pseudo-critical:2}
\sum_{1 \leq k_1 \leq K}
\norm*{
	\Phi_{\zeta_r^{k_1}}
	\parens*{
		\abs*{\sum_{1 \leq k_2 \leq K} \delta_{k_2} \Phi_{\zeta_r^{k_2}}}^{p - 2}
		-
		\abs{W_{\eta, r}}^{p - 2}
	}
}_{L^{p'}}
\end{equation}
and
\begin{equation}
\label{pseudo-critical:3}
\frac{\Phi \parens{r}}{\eta}
\norm*{G \abs{W_{\eta, r}}^{p - 2}}_{L^{p '}}.
\end{equation}

\paragraph{Estimation of
$\eqref{pseudo-critical:1}$.} In view of Item \eqref{elementary:1} in Lemma \ref{lem:elementary},
\[
\abs*{
	\Phi_{\zeta^{k_1}_r}^{p - 2}
	-
	\abs*{\sum_{1 \leq k_2 \leq K} \delta_{k_2} \Phi_{\zeta^{k_2}_r}}^{p - 2}
}
\leq
\sum_{\substack{1 \leq k_2 \leq K; \\ k_2 \neq k_1}} \Phi_{\zeta^{k_2}_r}^{p - 2}.
\]
In view of Corollary \ref{cor:entire-space}, there exists $\eta_0 > 1$ such that
\[
\sum_{\substack{1 \leq k_2 \leq K; \\ k_2 \neq k_1}}
	\norm*{\Phi_{\parens{\zeta^{k_1}_r - \zeta^{k_2}_r}} \Phi^{p - 2}}_{L^{p '}}^{p '}
\lesssim
\frac
	{r^{\frac{\parens{3 - p} N + p - 1}{2}}}
	{\eta^3},
\]
and so
$
\eqref{pseudo-critical:1}
\lesssim
r^{\frac{\parens{3 - p} N + p - 1}{2 p '}}
\eta^{- \frac{3}{p '}}
$
for every $\eta \geq \eta_0$ and $r \in R_\eta$.

\paragraph{Estimation of $\eqref{pseudo-critical:2}$.}
Analogously,
\begin{align*}
\abs*{
	\abs*{\sum_{1 \leq k \leq K} \delta_k \Phi_{\zeta^k_r}}^{p - 2}
	-
	\abs{W_{\eta, r}}^{p - 2}
}
&\leq
\sum_{1 \leq k \leq K}
	\parens*{\frac{\Phi \parens{\zeta^k_r}}{\beta_\eta \parens{1}}}^{p - 2}
	G^{p - 2}
;
\\
&\leq
K
\parens*{
	\frac{\Phi \parens{r}}{\beta_\eta \parens{1}}
}^{p - 2}
G^{p - 2}
.
\end{align*}
Therefore,
\[
\eqref{pseudo-critical:2}
\leq
K
\parens*{
	\frac{\Phi \parens{r}}{\beta_\eta \parens{1}}
}^{p - 2}
\parens*{
	\sum_{1 \leq k \leq K}
	\norm{\Phi_{\zeta^k_r} G^{p - 2}}_{L^{p '}}
}.
\]
Due to Corollary \ref{cor:entire-space}, there exists $\eta_0 > 1$ for which
\[
\norm{\Phi_{\zeta^k_r} G^{p - 2}}_{L^{p '}}^{p '}
\lesssim
\frac
	{r^{\frac{\parens{3 - p} N + p - 1}{2}}}
	{\eta},
\]
and so
\[
\eqref{pseudo-critical:2}
\lesssim
\frac{1}
	{
		r^{
			\frac
				{\parens{p - 2} \parens{N - 1}}
				{2}
		}
		\eta^{2 \parens{p - 2}}
	}
\times
\frac
	{r^{\frac{\parens{3 - p} N + p - 1}{2 p '}}}
	{
		\eta^{\frac{1}{p '}}
	}
\]
for every $\eta \geq \eta_0$ and $r \in R_\eta$.

\paragraph{Estimation of $\eqref{pseudo-critical:3}$.}
It follows from the definition of $W_{\eta, r}$ that
\[
\norm*{
	G \abs{W_{\eta, r}}^{p - 2}
}_{L^{p '}}
=
\norm*{
	G
	\abs*{
		\sum_{1 \leq k \leq K}
		\parens*{
			\delta_k \Phi_{\zeta^k_r}
			+
			\delta_k
			\frac
				{\Phi \parens{\zeta^k_r}}
				{\beta_\eta \parens{1}}
			G
		}
	}^{p - 2}
}_{L^{p '}}.
\]
In view of Item \eqref{elementary:1} in Lemma \ref{lem:elementary},
\begin{align*}
\norm*{
	G \abs{W_{\eta, r}}^{p - 2}
}_{L^{p '}}
&\leq
\sum_{1 \leq k \leq K}
\norm*{
	G
	\parens*{
		\Phi_{\zeta^k_r}^{p - 2}
		+
		\parens*{
			\frac
				{\Phi \parens{\zeta^k_r}}
				{\beta_\eta \parens{1}}
		}^{p - 2}
		G^{p - 2}
	}
}_{L^{p '}};
\\
&\leq
\sum_{1 \leq k \leq K}
\parens*{
	\norm*{
		G
		\Phi_{\zeta^k_r}^{p - 2}
	}_{L^{p '}}
	+
	\parens*{
		\frac
			{\Phi \parens{\zeta^k_r}}
			{\beta_\eta \parens{1}}
	}^{p - 2}
	\norm{G}^{p - 1}_{L^p}
}.
\end{align*}
Therefore,
\[
\eqref{pseudo-critical:3}
\lesssim
\frac{1}{r^{\frac{N - 1}{2}} \eta^2}
\times
\frac
	{r^{\frac{\parens{3 - p} N + p - 1}{2 p '}}}
	{
		\eta^{\frac{1}{p '}}
	}
+
\frac{1}
	{
		r^{
			\frac
				{\parens{p - 1} \parens{N - 1}}
				{2}
		}
		\eta^{2 \parens{p - 1}}
	}
\]
for every $\eta \geq \eta_0$ and $r \in R_\eta$.
\end{proof}

\subsection{An ansatz for the solutions, the auxiliary and bifurcation equations}

Our strategy consists in looking for solutions to the critical point equation
\[\nabla S_\eta \parens{u} = 0; \quad u \in H^1_\eta\]
with an \textit{a priori} known form. That is, we will actually solve
\begin{equation}
\label{eqn:critical-point-equation}
\nabla S_\eta \parens{W_{\eta, r} + \nu} = 0;
\quad
\nu \in \N_{\eta, r},
\end{equation}
where
\begin{multline*}
\N_{\eta, r}
:=
\left\{
	\phi + q G:
	\phi \in H^1; \quad
	q \in \real
	\quad \text{and} \quad
	\angles{\phi \mid \partial_i \Phi_{\zeta_r^k}}_{H^1} = 0
\right.
\\
\left.
	\text{for every} \quad
	\parens{i, k} \in \set{1, \ldots, N} \times \set{1, \ldots, K}
\right\}
\subset
H^1_\eta.
\end{multline*}
Instead of directly trying to solve \eqref{eqn:critical-point-equation}, we will consider its orthogonal decomposition that follows:
\begin{equation}
\label{eqn:system}
\begin{cases}
\Pi_{\eta, r} \brackets{\nabla S_\eta \parens{W_{\eta, r} + \nu}}
=
0
&\text{\emph{(auxiliary equation)}};
\\
\parens{\mathrm{id}_{H^1_\eta} - \Pi_{\eta, r}}
\brackets{\nabla S_\eta \parens{W_{\eta, r} + \nu}}
=
0
&\text{\emph{(bifurcation equation)}};
\\
\nu \in \N_{\eta, r},
\end{cases}
\end{equation}
where $\Pi_{\eta, r} \colon H^1_\eta \to \N_{\eta, r}$ denotes the corresponding $H^1_\eta$-orthogonal projection.

\subsection{Solving the auxiliary equation}
\label{sect:auxiliary}

In this section, we accomplish the first step to solve \eqref{eqn:system}, which consists in showing that if $\eta$ is sufficiently large, then we can associate each
$r \in R_\eta$ with a solution to the corresponding auxiliary equation as stated in the result that follows.

\begin{lem}
\label{lem:implicit-function}
There exists $\eta_0 > 1$ such that we can associate every $\eta \geq \eta_0$ with a mapping of class $C^1$
\[
R_\eta \ni r \mapsto \nu_{\eta, r} \in H^1_\eta
\]
such that $\nu_{\eta, r} \in \N_{\eta, r}$ and
\[
\Pi_{\eta, r}
\brackets*{
	\nabla S_\eta \parens{W_{\eta, r}
	+
	\nu_{\eta, r}}
}
=
0
\]
for every $r \in R_\eta$. Moreover, the so-obtained family of mappings is such that
\[
\norm{\nu_{\eta, r}}_{H^1_\eta}
\lesssim
\norm*{\nabla S_\eta \parens{W_{\eta, r}}}_{H^1_\eta}
\]
for every $\eta \geq \eta_0$ and $r \in R_\eta$.
\end{lem}

The proof of the lemma is not interesting by itself, but it depends on a few technical results whose validity is not obvious in the present context (Lemmas \ref{lem:uniform-inversion}, \ref{lem:estimate-second-derivative} and \ref{lem:fixed-point}), so we develop them in what follows. More precisely, Lemma \ref{lem:implicit-function} is proved with classical arguments as in \cite[Proof of Proposition 8.7]{ambrosettiPerturbationMethodsSemilinear2006} or \cite[Proof of Lemma 3.3]{depaularamosClusterSemiclassicalStates2023}.

\subsubsection{Uniform inversion of $L_{\eta, r}$}

The goal of this section is to prove the result that follows, where
$
L_{\eta, r}
\colon
\N_{\eta, r} \to \N_{\eta, r}
$
denotes the linear mapping defined as
\[
L_{\eta, r} \brackets{\nu}
=
\Pi_{\eta, r} \circ I_\eta^{- 1} \brackets*{
	S_\eta'' \parens{W_{\eta, r}} \brackets{\nu, \cdot}
}
\]
and $I_\eta \colon H^1_\eta \to H^{- 1}_\eta$ denotes the canonical linear isometry.

\begin{lem}
\label{lem:uniform-inversion}
There exist $\eta_0 > 1$ and $\bar{C} > 0$ such that
\begin{center}
\begin{enumerate*}
\item the linear operator
$L_{\eta, r}$ is invertible;
\quad \quad
\item
$\norm{L_{\eta, r}^{-1}} \leq \bar{C}$
\end{enumerate*}
\end{center}
for every $\eta \geq \eta_0$ and $r \in R_\eta$.
\end{lem}

We need three preliminary results to prove the lemma. Let us recall a few well-known properties of $S_\infty'' \parens{\Phi}$.

\begin{lem}[{\cite[Lemma 8.6]{ambrosettiPerturbationMethodsSemilinear2006}}] \label{lem:AmbrosettiMalchiodi}
The following statements hold:
\begin{enumerate}
\item
$S_\infty'' \parens{\Phi} \brackets{\Phi, \Phi}
=
\parens{1 - p} \norm{\Phi}_{H^1}^2
<
0$;
\item
$
\ker S_\infty'' \parens{\Phi}
=
\vspan \set{\partial_i \Phi : i \in \set{1, \ldots, N}}
$;
\item
$S_\infty'' \parens{\Phi} \brackets{w, w} \gtrsim \norm{w}_{H^1}^2$ for every
$w \perp \parens{\ker S_\infty'' \parens{\Phi} \oplus \vspan \set{\Phi}}$.
\end{enumerate}
\end{lem}

We need to introduce a preliminary notation to state the following collection of estimates. Given $z \in \sphere^{N - 1} \subset \real^N$, we define
\[
\dot{\Psi}_{\eta, y}^z
=
\lim_{t \to 0}
\parens*{
	t^{- 1}
	\parens{
		\Psi_{\eta, y + t z}
		-
		\Psi_{\eta, y}
	}
}
\in
H^1_\eta.
\]

\begin{lem}
\label{lem:estimates-exponential}
There exists $\eta_0 \in \ooi{1, \infty}$ such that
\[
\abs*{
	\angles{
		\Psi_{\eta, \zeta_r^{k_1}}
		\mid
		\Psi_{\eta, \zeta_r^{k_2}}
	}_{H^1_\eta}
},
\abs*{
	\angles{
		\dot{\Psi}_{\eta, \zeta_r^{k_1}}^{z_1}
		\mid
		\Psi_{\eta, \zeta_r^{k_2}}
	}_{H^1_\eta}
},
\abs*{
	\angles{
		\dot{\Psi}_{\eta, \zeta_r^{k_1}}^{z_1}
		\mid
		\dot{\Psi}_{\eta, \zeta_r^{k_2}}^{z_2}
	}_{H^1_\eta}
}
\lesssim
\frac
	{r^{\frac{N - 1 + \parens{3 - p} N}{2}}}
	{\eta^3}
\]
for every $z_1, z_2 \in \sphere^{N - 1}$,
$\eta \in \coi{\eta_0, \infty}$,
$r \in R_\eta$ and $k_1 \neq k_2$ in $\set{1, \ldots, K}$.
\end{lem}
\begin{proof}
~\paragraph{First estimate.} Clearly,
\[
\angles{
	\Psi_{\eta, \zeta_r^{k_1}} \mid \Psi_{\eta, \zeta_r^{k_2}}
}_{H^1_\eta}
=
\angles{\Phi_{\zeta_r^{k_1} - \zeta_r^{k_2}} \mid \Phi}_{H^1}
+
\frac{\Phi \parens{\zeta_r^{k_1}} \Phi \parens{\zeta_r^{k_2}}}{\beta_\eta \parens{1}}.
\]
Due to the equality $- \Delta \Phi + \Phi = \Phi^{p - 1}$ and Corollary \ref{cor:entire-space}, there exists $\eta_0 > 1$ such that
\[
\abs*{\angles{\Phi_{\zeta_r^{k_1} - \zeta_r^{k_2}} \mid \Phi}_{H^1}}
=
\abs*{\int \Phi_{\zeta_r^{k_1} - \zeta_r^{k_2}} \Phi^{p - 1} \dif \Lebesgue}
\lesssim
\frac
	{r^{\frac{N - 1 + \parens{3 - p} N}{2}}}
	{\eta^3}
\]
for every $\eta \geq \eta_0$ and $r \in R_\eta$.

\paragraph{Second estimate.} It suffices to estimate
$
\angles{
	\dot{\Psi}_{\eta, \zeta_r^{k_1}}^{e_i} \mid \Psi_{\eta, \zeta_r^{k_2}}
}_{H^1_\eta}
$
for every $i \in \set{1, \ldots, N}$, where $e_1, \ldots, e_N$ denotes the canonical basis of $\real^N$. It is easy to verify that
\[
\dot{\Psi}_{\eta, \zeta_r^{k_1}}^{e_j}
=
-
\partial_j \Phi_{\zeta_r^{k_1}}
+
\frac{\partial_j \Phi \parens{\zeta_r^{k_1}}}{\beta_\eta \parens{1}}
G.
\]
As
$
- \Delta \parens{\partial_j \Phi_{\zeta_r^{k_1}}} + \partial_j \Phi_{\zeta_r^{k_1}}
=
\parens{p - 1}
\parens{\partial_j \Phi_{\zeta_r^{k_1}}}
\Phi_{\zeta_r^{k_1}}^{p - 2}
$, we deduce that
\[
\angles{
	\partial_j \Phi_{\zeta_r^{k_1}}
	\mid
	\Psi_{\eta, \zeta_r^{k_2}}
}_{H^1_\eta}
=
\angles{
	\partial_j \Phi_{\zeta_r^{k_1}}
	\mid
	\Phi_{\zeta_r^{k_2}}
}_{H^1}
=
\parens{p - 1}
\int
	\Phi_{\zeta_r^{k_2} - \zeta_r^{k_1}}
	\parens{\partial_j \Phi} \Phi^{p - 2}
\dif \Lebesgue.
\]
We can obtain an upper bound for this integral by arguing as in the previous paragraph.

\paragraph{Third estimate.} Similar to the previous proofs.
\end{proof}

Recall the following corollary of the Lax--Milgram Theorem (see \cite[Lemma 3.3]{ianniConcentrationPositiveBound2008}).
\begin{lem} \label{lem:LaxMilgram}
Let $H$ be a Hilbert space and let $L \colon H \to H$ be a self-adjoint linear operator. Suppose that
\begin{enumerate}
\item
$A$ is a finite-dimensional linear subspace of $H$,
\item
$B$ is a linear subspace of $H$,
\item
$H = A \oplus B$,
\item
there exists $c > 0$ such that
\begin{enumerate}
\item
$\angles{L h \mid h}_H \leq -c \norm{h}_H^2$ for every $h \in A$ and
\item
$\angles{L h \mid h}_H \geq c \norm{h}_H^2$ for every $h \in B$.
\end{enumerate}
\end{enumerate}
Then $L$ is invertible and $\norm{L^{-1}} \leq 1/c$.
\end{lem}

We proceed to the proof of Lemma \ref{lem:uniform-inversion}.
\begin{proof}
[Proof of Lemma \ref{lem:uniform-inversion}]
Due to \eqref{eqn:trick},
\[
S_\eta'' \parens{W_{\eta, r}} \brackets{\nu, \nu}
=
S_\infty'' \parens*{\sum_{1 \leq k \leq K} \delta_k \Phi_{\zeta_r^k}}
\brackets{\phi, \phi}
+
\eqref{uniform:1}
+
\beta_\eta \parens{1} \abs{q}^2
\]
for every $\nu = \phi + q G \in \N_{\eta, r}$, where
\begin{equation}
\label{uniform:1}
\int
	\parens*{
		\abs*{\sum_{1 \leq k \leq K} \delta_k \Phi_{\zeta_r^k}}^{p - 2}
		-
		\abs{W_{\eta, r}}^{p - 2}
	}
	\abs{\nu}^2
\dif \Lebesgue.
\end{equation}
Due to Item \eqref{elementary:1} in Lemma \ref{lem:elementary}, there exists
$\eta_0 > 1$ such that
\[
\abs*{\eqref{uniform:1}}
\lesssim
\parens*{\frac{\Phi \parens{r}}{\eta}}^{p - 2}
\norm{\nu}_{H^1_\eta}^2
\lesssim
\frac{1}{\eta^{2 \parens{p - 2}}}
\norm{\nu}_{H^1_\eta}^2
\]
for every $\eta \geq \eta_0$, $r \in R_\eta$ and $\nu \in H^1_\eta$.

In view of this estimate, we define the bilinear form
$T_{\eta, r} \colon \N_{\eta, r} \times \N_{\eta, r} \to \real$
as
\[
T_{\eta, r} \brackets{\nu, \nu}
=
S_\infty'' \parens*{\sum_{1 \leq k \leq K} \delta_k \Phi_{\zeta_r^k}}
\brackets{\phi, \phi}
+
\beta_\eta \parens{1} \abs{q}^2
\]
for every $\nu = \phi + q G \in \N_{\eta, r}$. We also let
$
\Pi_{\infty, r}
\colon
H^1 \to C_r^{\perp_{H^1}}
$
denote the $H^1$-orthogonal projection over $C_r^{\perp_{H^1}}$, where
\[
C_r
:=
\set*{
	\partial_i \Phi_{\zeta_r^k}:
	\parens{i, k} \in \set{1, \ldots, N} \times \set{1, \ldots, K}
}
\subset
H^1.
\]
Notice that
$\Pi_{\eta, r}|_{H^1} = \Pi_{\infty, r}$, where we recall that
$\Pi_{\eta, r}$ denotes the
$H^1_\eta$-orthogonal projection over
$\N_{\eta, r}$. Indeed, this equality follows from the uniqueness of orthogonal projections and the fact that
$\norm{\phi}_{H^1_\eta} = \norm{\phi}_{H^1}$
for every $\phi \in H^1$.

Now, we define
\[
A_r
=
\mathrm{span} \set*{\Pi_{\infty, r} \brackets{\Phi_{\zeta_r^k}}}_{1 \leq k \leq K}
\subset
H^1
\]
and
\[
B_{\eta, r} = \parens{A_r \cup C_r}^{\perp_{H^1}} \oplus \vspan \set{G} \subset H^1_\eta,
\]
so that
$\N_{\eta, r} = A_r \oplus B_{\eta, r}$
(compare with \cite[p. 260]{ruizClusterSolutionsSchrodingerPoissonSlater2011}). Due to Lemma \ref{lem:LaxMilgram}, we only have to prove that there exists $\eta_0 > 1$ such that
\[
T_{\eta, r} \brackets{\nu, \nu} \gtrsim \norm{\nu}_{H^1_\eta}^2
\quad \text{and} \quad
- T_{\eta, r} \brackets{\nu', \nu'} \gtrsim \norm{\nu'}_{H^1_\eta}^2
\]
for every $\eta \geq \eta_0$, $r \in R_\eta$,
$\nu \in B_{\eta, r}$ and
$\nu' \in A_r$. The first estimate follows from the computations in the proof of \cite[Item (b) in Lemma 3.4]{ruizClusterSolutionsSchrodingerPoissonSlater2011}. In view of Lemmas \ref{lem:AmbrosettiMalchiodi} and \ref{lem:estimates-exponential}, the second estimate may be proved as \cite[Item (a) in Lemma 3.4]{ruizClusterSolutionsSchrodingerPoissonSlater2011} or \cite[Lemma 3.7]{depaularamosClusterSemiclassicalStates2023}.
\end{proof}

\subsubsection{An estimate for the variation of $S_\eta''$ and the fixed point result}

The next lemma furnishes an important estimate of the variation of $S_\eta''$ in a neighborhood of the pseudo-critical point $W_{\eta, r}$.	

\begin{lem}
\label{lem:estimate-second-derivative}
The inequality
\[
\norm*{
	S_\eta'' \parens{W_{\eta, r} + u} \brackets{v, \cdot}
	-
	S_\eta'' \parens{W_{\eta, r}} \brackets{v, \cdot}
}_{H^{-1}_\eta}
\leq
2
\norm{u}^{p - 2}_{H^1_\eta} \norm{v}_{H^1_\eta}
\]
holds for every $\eta > 1$, $r \in R_\eta$ and $u, v \in H^1_\eta$.
\end{lem}
\begin{proof}
It is clear that
\begin{multline*}
S_\eta'' \parens{W_{\eta, r} + u} \brackets{v, w}
-
S_\eta'' \parens{W_{\eta, r}} \brackets{v, w}
=
\\
=
-
\parens{p - 1}
\int
	\parens*{\abs{W_{\eta, r} + u}^{p - 2} - \abs{W_{\eta, r}}^{p - 2}} v w
\dif \Lebesgue.
\end{multline*}
We can estimate this difference with Item \eqref{elementary:1} in Lemma \ref{lem:elementary}.
\end{proof}

In view of Lemmas \ref{lem:pseudo-critical}, \ref{lem:uniform-inversion} and \ref{lem:estimate-second-derivative}, we can prove the following result with standard arguments as in \cite[p. 264]{ruizClusterSolutionsSchrodingerPoissonSlater2011} or as in the proof of \cite[Lemma 3.9]{depaularamosClusterSemiclassicalStates2023}.

\begin{lem}
\label{lem:fixed-point}
Suppose that $\eta_0 > 1$ and $\bar{C} > 0$ satisfy the hypothesis of Lemma \ref{lem:uniform-inversion}. Then given
$\eta \geq \eta_0$ and $r \in R_\eta$, the following problem has a unique solution:
\[
\begin{cases}
\Pi_{\eta, r} \brackets{\nabla S_\eta \parens{W_{\eta, r} + \nu}} = 0;
\\
\nu \in \mathcal{M}_{\eta, r, \bar{C}},
\end{cases}
\]
where
\[
\mathcal{M}_{\eta, r, \bar{C}}
:=
\set*{
	\nu \in \N_{\eta, r} :
	\norm{\nu}_{H^1_\eta}
	\leq
	2 \bar{C} \norm*{\nabla S_\eta \parens{W_{\eta, r}}}_{H^1_\eta}
}.
\]
\end{lem}

\subsection{Solving the bifurcation equation}

The goal of this section is to establish the existence of solution to \eqref{eqn:system} for sufficiently large $\eta$. We begin with a preliminary definition.

\begin{defn}
If $\eta_0$ satisfies the hypothesis in Lemma \ref{lem:implicit-function}, then we associate each $\eta \geq \eta_0$ with the \emph{reduced functional}
$\sigma_\eta \in C^1 \parens{R_\eta, \real}$ defined as
\[
\sigma_\eta \parens{r}
=
S_\eta \parens{W_{\eta, r} + \nu_{\eta, r}}.
\]
\end{defn}

The fundamental property of the reduced functional is that if $\eta$ is sufficiently large, then critical points of $\sigma_\eta$ are naturally associated with critical points of $S_\eta$ as stated in the next result.

\begin{lem}
\label{lem:natural-constraint}
There exists $\eta_0 > 1$ such that if
$\eta \geq \eta_0$, then
$\sigma_\eta' \parens{r} = 0$ if, and only if,
$\nabla S_\eta \parens{W_{\eta, r} + \nu_{\eta, r}} = 0$.
\end{lem}

We omit the proof of Lemma \ref{lem:natural-constraint} because it follows from standard arguments which use Lemmas \ref{lem:implicit-function}, \ref{lem:estimates-exponential} (for details, see the proof of \cite[Lemma 3.11]{depaularamosClusterSemiclassicalStates2023}). It is easy to verify that $\sigma_\eta'$ is bounded, so $\sigma_\eta \colon R_\eta \to \real$ admits a unique continuous extension, which we denote by
$
\overline{\sigma_\eta}
\colon
\overline{R_\eta} \to \real
$.

At this point, we want to develop an expansion of $\overline{\sigma_\eta}$. As a preliminary step, we will expand
\[
\overline{R_\eta}
\ni
r \mapsto S_\eta \parens{W_{\eta, r}}.
\]

\begin{lem}
\label{lem:expansion}
Let $F_\eta \colon \overline{R_\eta} \to \real$ be given by
\[
F_\eta \parens{r}
=
K C_0
-
\frac{\Phi \parens{r}^2}{2 \eta}
+
\chi
\int \Phi_{3 r} \Phi^{p - 1} \dif \Lebesgue,
\]
where
\[
C_0
:=
S_\infty \parens{\Phi}
=
\frac{1}{2} \norm{\Phi}_{H^1}^2 - \frac{1}{p} \norm{\Phi}_{L^p}^p
=
\frac{p - 2}{2 p} \norm{\Phi}_{H^1}^2
>
0
\]
and we recall that the positive constant $\chi$ was defined in \eqref{eqn:chi}. Then there exists $\eta_0 > 1$ such that
\[
\abs*{
	S_\eta \parens{W_{\eta, r}}
	-
	F_\eta \parens{r}
}
\lesssim
\begin{cases}
\displaystyle
\eta^{- 4},
&\text{if} ~ K \in \set{2, 3};
\\
\\
\displaystyle
\frac
	{r^{\frac{p - 1 + \parens{3 - p} N}{2}}}
	{\eta^{\min \parens{\ell, 4}}},
&\text{if} ~ K \geq 4
\end{cases}
\]
for every $\eta \geq \eta_0$ and
$r \in \overline{R_\eta}$, where
\[
\ell
:=
3
\abs{e^{\frac{4 \pi}{K} \iu} - 1}
=
3
\sqrt{
	\parens*{\cos \parens*{\frac{4 \pi}{K}} - 1}^2 + \sin \parens*{\frac{4 \pi}{K}}^2
}
\]
is strictly greater than 3 when $K \geq 4$.
\end{lem}
\begin{proof}
By definition,
\begin{equation}
\label{expansion:-2}
S_\eta \parens{W_{\eta, r}}
=
\frac{1}{2} \norm{W_{\eta, r}}_{H^1_\eta}^2
-
\frac{1}{p} \norm{W_{\eta, r}}_{L^p}^p.
\end{equation}
Clearly,
\begin{equation}
\label{expansion:-1}
\frac{1}{2} \norm{W_{\eta, r}}_{H^1_\eta}^2
=
\frac{1}{2}
\norm*{
	\sum_{1 \leq k \leq K} \parens{\delta_k \Psi_{\eta, \zeta_r^k}}
}_{H^1_\eta}^2
=
\eqref{expansion:0}
+
\eqref{expansion:1},
\end{equation}
where
\begin{equation}
\label{expansion:0}
\frac{1}{2} \sum_{1 \leq k \leq K} \norm{\Psi_{\eta, \zeta_r^k}}_{H^1_\eta}^2
\end{equation}
and
\begin{equation}
\label{expansion:1}
\frac{1}{2}
\sum_{\substack{1 \leq k_1, k_2 \leq K; \\ k_1 \neq k_2}}
	\parens*{
		\delta_{k_1} \delta_{k_2}
		\angles{\Psi_{\eta, \zeta_r^{k_1}} \mid \Psi_{\eta, \zeta_r^{k_2}}}_{H^1_\eta}
	}.
\end{equation}
On one hand,
\begin{equation}
\label{expansion:1.1}
\eqref{expansion:0}
=
\frac{K}{2} \norm{\Phi}_{H^1}^2
+
\frac{1}{2 \beta_\eta \parens{1}} \sum_{1 \leq k \leq K} \Phi \parens{\zeta_r^k}^2.
\end{equation}
On the other hand, the equality $- \Delta \Phi + \Phi = \Phi^{p - 1}$ implies
\begin{equation}
\label{expansion:1.2}
\eqref{expansion:1}
=
\frac{1}{2}
\sum_{\substack{1 \leq k_1, k_2 \leq K; \\ k_1 \neq k_2}}
\parens*{
	\delta_{k_1} \delta_{k_2}
	\int \Phi_{\zeta_r^{k_1} - \zeta_r^{k_2}} \Phi^{p - 1} \dif \Lebesgue
	+
	\frac{\delta_{k_1} \delta_{k_2}}{\beta_\eta \parens{1}}
	\Phi \parens{\zeta_r^{k_1}} \Phi \parens{\zeta_r^{k_2}}
}.
\end{equation}
In view of \eqref{expansion:-2}, \eqref{expansion:-1}, \eqref{expansion:1.1} and \eqref{expansion:1.2}, we obtain
\begin{multline*}
S_\eta \parens{W_{\eta, r}}
=
\frac{K}{2} \norm{\Phi}_{H^1}^2
+
\frac{1}{2 \beta_\eta \parens{1}}
\sum_{1 \leq k_1, k_2 \leq K}
	\parens*{
		\delta_{k_1} \delta_{k_2} \Phi \parens{\zeta_r^{k_1}} \Phi \parens{\zeta_r^{k_2}}
	}
+
\\
+
\frac{1}{2}
\sum_{\substack{1 \leq k_1, k_2 \leq K; \\ k_1 \neq k_2}}
\int
	\delta_{k_1} \delta_{k_2}
	\Phi_{\zeta_r^{k_1} - \zeta_r^{k_2}} \Phi^{p - 1}
\dif \Lebesgue
-
\frac{1}{p} \norm{W_{\eta, r}}_{L^p}^p.
\end{multline*}

By summing and subtracting
$\frac{K}{p} \norm{\Phi}_{L^p}^p$, we obtain
\begin{multline*}
S_\eta \parens{W_{\eta, r}}
=
K C_0
+
\frac{1}{2 \beta_\eta \parens{1}}
\sum_{1 \leq k_1, k_2 \leq K}
	\parens*{
		\delta_{k_1} \delta_{k_2} \Phi \parens{\zeta_r^{k_1}} \Phi \parens{\zeta_r^{k_2}}
	}
+
\\
+
\frac{K}{p} \norm{\Phi}_{L^p}^p
+
\frac{1}{2}
\sum_{\substack{1 \leq k_1, k_2 \leq K; \\ k_1 \neq k_2}}
\int
	\delta_{k_1} \delta_{k_2}
	\Phi_{\zeta_r^{k_1} - \zeta_r^{k_2}} \Phi^{p - 1}
\dif \Lebesgue
-
\frac{1}{p} \norm{W_{\eta, r}}_{L^p}^p.
\end{multline*}
We sum and subtract
\[
\frac{2}{p}
\sum_{\substack{1 \leq k_1, k_2 \leq K; \\ k_1 \neq k_2}}
	\parens*{
		\delta_{k_1} \delta_{k_2}
		\int \Phi_{\zeta_r^{k_1} - \zeta_r^{k_2}} \Phi^{p - 1} \dif \Lebesgue
	}
\]
to deduce that
\begin{multline*}
S_\eta \parens{W_{\eta, r}}
=
K C_0
+
\frac{1}{2 \beta_\eta \parens{1}}
\sum_{1 \leq k_1, k_2 \leq K}
	\parens*{
		\delta_{k_1} \delta_{k_2} \Phi \parens{\zeta_r^{k_1}} \Phi \parens{\zeta_r^{k_2}}
	}
+
\\
-
\frac{4 - p}{2 p}
\sum_{\substack{1 \leq k_1, k_2 \leq K; \\ k_1 \neq k_2}}
\int
	\delta_{k_1} \delta_{k_2}
	\Phi_{\zeta_r^{k_1} - \zeta_r^{k_2}} \Phi^{p - 1}
\dif \Lebesgue
+
\\
+
\frac{K}{p} \norm{\Phi}_{L^p}^p
+
\frac{2}{p}
\sum_{\substack{1 \leq k_1, k_2 \leq K; \\ k_1 \neq k_2}}
\int
	\delta_{k_1} \delta_{k_2}
	\Phi_{\zeta_r^{k_1} - \zeta_r^{k_2}} \Phi^{p - 1}
\dif \Lebesgue
-
\frac{1}{p} \norm{W_{\eta, r}}_{L^p}^p.
\end{multline*}
By summing and subtracting
$
\frac{1}{p}
\norm{
	\sum_{1 \leq k \leq K}
	\parens{\delta_k \Phi_{\zeta_r^k}}
}_{L^p}^p
$,
we see that
\begin{multline*}
S_\eta \parens{W_{\eta, r}}
=
K C_0
+
\frac{1}{2 \beta_\eta \parens{1}}
\sum_{1 \leq k_1, k_2 \leq K}
	\parens*{
		\delta_{k_1} \delta_{k_2} \Phi \parens{\zeta_r^{k_1}} \Phi \parens{\zeta_r^{k_2}}
	}
+
\\
-
\frac{4 - p}{2 p}
\sum_{\substack{1 \leq k_1, k_2 \leq K; \\ k_1 \neq k_2}}
\int
	\delta_{k_1} \delta_{k_2}
	\Phi_{\zeta_r^{k_1} - \zeta_r^{k_2}} \Phi^{p - 1}
\dif \Lebesgue
+
\\
+
\frac{1}{p}
\int
	\sum_{1 \leq k_1 \leq K}
		\parens*{
			\Phi_{\zeta_r^{k_1}} \parens{x}^p
			+
			2
			f \parens*{\delta_{k_1} \Phi_{\zeta_r^{k_1}} \parens{x}}	
			\sum_{\substack{1 \leq k_2 \leq K; \\ k_1 \neq k_2}}
				\parens*{\delta_{k_2} \Phi_{\zeta_r^{k_2}} \parens{x}}
		}
	-
\\
	-
	\abs*{\sum_{1 \leq k \leq K} \delta_k \Phi_{\zeta_r^k} \parens{x}}^p
\dif \Lebesgue \parens{x}
+
\frac{1}{p}
\norm*{
	\sum_{1 \leq k \leq K}
	\parens{\delta_k \Phi_{\zeta_r^k}}
}_{L^p}^p
-
\frac{1}{p} \norm{W_{\eta, r}}_{L^p}^p.
\end{multline*}
Finally, the sum and subtraction of
\[
\frac{1}{\beta_\eta \parens{1}}
\parens*{\sum_{1 \leq k \leq K} \delta_k \Phi \parens{\zeta_r^k}}
\int
	G \parens{x}
	f \parens*{\sum_{1 \leq k \leq K} \delta_k \Phi_{\zeta_r^k} \parens{x}}
\dif \Lebesgue \parens{x},
\]
gives us
\[
S_\eta \parens{W_{\eta, r}}
=
K C_0
+
\eqref{expansion:2}
+
\eqref{expansion:3}
+
\eqref{expansion:5}
+
\eqref{expansion:4},
\]
where
\begin{multline}
\label{expansion:2}
\frac{1}{2 \beta_\eta \parens{1}}
\sum_{1 \leq k_1, k_2 \leq K}
	\parens*{
		\delta_{k_1} \delta_{k_2} \Phi \parens{\zeta_r^{k_1}} \Phi \parens{\zeta_r^{k_2}}
	}
+
\\
-
\frac{1}{\beta_\eta \parens{1}}
\parens*{\sum_{1 \leq k \leq K} \delta_k \Phi \parens{\zeta_r^k}}
\int
	G \parens{x}
	f \parens*{\sum_{1 \leq k \leq K} \delta_k \Phi_{\zeta_r^k} \parens{x}}
\dif \Lebesgue \parens{x},
\end{multline}
\begin{equation}
\label{expansion:3}
- \frac{4 - p}{2 p}
\sum_{\substack{1 \leq k_1, k_2 \leq K; \\ k_1 \neq k_2}}
	\parens*{
		\delta_{k_1} \delta_{k_2}
		\int \Phi_{\zeta_r^{k_1} - \zeta_r^{k_2}} \Phi^{p - 1} \dif \Lebesgue
	},
\end{equation}
\begin{multline}
\label{expansion:5}
\frac{1}{p}
\int
	\sum_{1 \leq k_1 \leq K}
		\parens*{
			\Phi_{\zeta_r^{k_1}} \parens{x}^p
			+
			2
			f \parens*{\delta_{k_1} \Phi_{\zeta_r^{k_1}} \parens{x}}	
			\sum_{\substack{1 \leq k_2 \leq K; \\ k_1 \neq k_2}}
				\parens*{\delta_{k_2} \Phi_{\zeta_r^{k_2}} \parens{x}}
		}
	-
\\
	-
	\abs*{\sum_{1 \leq k \leq K} \delta_k \Phi_{\zeta_r^k} \parens{x}}^p
\dif \Lebesgue \parens{x}
\end{multline}
and
\begin{multline}
\label{expansion:4}
\frac{1}{p}
\int
	\abs*{\sum_{1 \leq k \leq K} \delta_k \Phi_{\zeta_r^k} \parens{x}}^p
	+
\\
	+
	\frac{p}{\beta_\eta \parens{1}}
	\parens*{\sum_{1 \leq k \leq K} \delta_k \Phi \parens{\zeta_r^k}}
	G \parens{x}
	f \parens*{\sum_{1 \leq k \leq K} \delta_k \Phi_{\zeta_r^k} \parens{x}}
	-
\\
	-
	\abs{W_{\eta, r} \parens{x}}^p
\dif \Lebesgue \parens{x}.
\end{multline}

\paragraph{Expansion of $\eqref{expansion:2}$.}
Let us show that there exists $\eta_0 > 1$ such that
\[
\abs*{
	\eqref{expansion:2}
	+
	\frac{\Phi \parens{r}^2}{2 \eta}
}
\lesssim
\frac{1}{\eta^5}
\]
for every $\eta \geq \eta_0$ and $r \in R_\eta$. It is clear that
$
\eqref{expansion:2}
=
\eqref{expansion:2.1} - \eqref{expansion:2.2}
$,
where
\begin{equation}
\label{expansion:2.1}
\frac{1}{2 \beta_\eta \parens{1}}
\sum_{1 \leq k_1, k_2 \leq K}
	\parens*{
		\delta_{k_1} \delta_{k_2}
		\Phi \parens{\zeta_r^{k_1}}
		\Phi \parens{\zeta_r^{k_2}}
	}
\end{equation}
and
\begin{equation}
\label{expansion:2.2}
\frac{1}{\beta_\eta \parens{1}}
\parens*{\sum_{1 \leq k \leq K} \delta_k \Phi \parens{\zeta_r^k}}
\int
	G \parens{x}
	f \parens*{\sum_{1 \leq k \leq K} \delta_k \Phi_{\zeta_r^k} \parens{x}}
\dif \Lebesgue \parens{x}.
\end{equation}

\subparagraph{Expansion of \eqref{expansion:2.1}}
Due to the inequality
\[
\min_{k \in \set{1, \ldots, K - 1}} \abs{\zeta_r^k}
=
r
\sqrt{
	\parens*{3 \sin \parens*{\frac{K - 2}{2 K} \pi}}^2
	+
	\parens*{1 + 3 \cos \parens*{\frac{K - 2}{2 K} \pi}}^2
}
>
3 r,
\]
there exists $\eta_0 > 1$ such that
\[
\abs*{
	\eqref{expansion:2.1}
	-
	\frac{\Phi \parens{r}^2}{2 \eta}
}
\lesssim
\frac{\Phi \parens{3 r} \Phi \parens{r}}{\eta}
\lesssim
\frac{1}{\eta^5}
\]
for every $\eta \geq \eta_0$ and $r \in R_\eta$.

\subparagraph{Expansion of \eqref{expansion:2.2}}
A change of variable shows that
\[
\eqref{expansion:2.2}
=
\frac{1}{\beta_\eta \parens{1}}
\sum_{k_1, k_2 \in \set{1, \ldots, K}}
\int
	\delta_{k_1} \delta_{k_2} \Phi \parens{\zeta_r^{k_1}}
	G \parens{x + \zeta_r^{k_2}}
	\Phi \parens{x}^{p - 1}
\dif \Lebesgue.
\]
As $\Phi = \Phi^{p - 1} \ast G$ (see \cite[p. 386]{gidasSymmetryPositiveSolutions1981}), it follows that
\[
\eqref{expansion:2.2}
= 
\frac{1}{\beta_\eta \parens{1}}
\sum_{k_1, k_2 \in \set{1, \ldots, K}} \parens*{
	\delta_{k_1} \delta_{k_2} \Phi \parens{\zeta_r^{k_1}} \Phi \parens{\zeta_r^{k_2}}
}.
\]
As such, it suffices to argue as in the previous expansion to deduce that there exists
$\eta_0 > 1$ such that
\[
\abs*{
	\eqref{expansion:2.2}
	-
	\frac{\Phi \parens{r}^2}{\eta}
}
\lesssim
\frac{\Phi \parens{3 r} \Phi \parens{r}}{\eta}
\lesssim
\frac{1}{\eta^5}
\]
for every $\eta \geq \eta_0$ and $r \in R_\eta$.

\paragraph{Expansion of $\eqref{expansion:3}$.}
We want to prove that there exists $\eta_0 > 1$ such that
\[
\abs*{
	\eqref{expansion:3}
	-
	\chi
	\int \Phi_{3 r} \Phi^{p - 1} \dif \Lebesgue
}
\lesssim
\frac
	{r^{\frac{p - 1 + \parens{3 - p} N}{2}}}
	{\eta^\ell}
\]
for every $\eta \geq \eta_0$ and $r \in R_\eta$. It is clear that
$
\eqref{expansion:3}
=
\chi
\int \Phi_{3 r} \Phi^{p - 1} \dif \Lebesgue
$
when $K \in \set{2, 3}$, so suppose that
$K \geq 4$. Suppose that
$k_1, k_2 \in \set{1, \ldots, K}$ are such that
$\abs{\zeta^{k_1}_r - \zeta^{k_2}_r} > 3 r$. Due to Corollary \ref{cor:entire-space}, there exists $\eta_0 \in \ooi{1, \infty}$ such that
\[
\abs*{\int \Phi_{\zeta_r^{k_1} - \zeta_r^{k_2}} \Phi^{p - 1} \dif \Lebesgue}
\lesssim
\frac
	{r^{\frac{p - 1 + \parens{3 - p} N}{2}}}
	{\eta^\ell}
\]
for every $\eta \in \coi{\eta_0, \infty}$ and
$r \in R_\eta$, hence the result.

\paragraph{Estimation of $\abs{\eqref{expansion:5}}$.} In view of Item \eqref{elementary:2} in Lemma \ref{lem:elementary},
\begin{multline*}
\abs*{\eqref{expansion:5}}
\leq
\sum_{\substack{
	1 \leq k_1, k_2, k_3 \leq K;
	\\
	k_1 \neq k_3; k_2 \neq k_3
}}
\int
	\Phi_{\zeta_r^{k_1} - \zeta_r^{k_3}}
	\Phi_{\zeta_r^{k_2} - \zeta_r^{k_3}}
	\Phi^{p - 2}
\dif \Lebesgue
\leq
\\
\leq
\frac{1}{2}
\sum_{\substack{
	1 \leq k_1, k_2, k_3 \leq K;
	\\
	k_1 \neq k_3; k_2 \neq k_3
}}
\parens*{
\int
	\Phi_{\zeta_r^{k_2} - \zeta_r^{k_3}}^2
	\Phi^{p - 2}
\dif \Lebesgue
+
\int
	\Phi_{\zeta_r^{k_1} - \zeta_r^{k_3}}^2
	\Phi^{p - 2}
\dif \Lebesgue
}.
\end{multline*}
Corollary \ref{cor:entire-space} shows that there exists $\eta_0 > 1$ such that
\[
\abs*{\eqref{expansion:5}}
\lesssim
\frac
	{r^{\frac{N + 1}{2}}}
	{r^{\frac{\parens{p - 1} N}{2}}}
e^{- 6 r}
\lesssim
\frac
	{r^{\frac{N + 1}{2}}}
	{r^{\frac{\parens{p - 1} N}{2}}}
\eta^{- 6}
\]
for every $\eta \geq \eta_0$ and $r \in R_\eta$.

\paragraph{Estimation of $\abs{\eqref{expansion:4}}$.}
Due to Item \eqref{elementary:4} in Lemma \ref{lem:elementary}, there exists $\eta_0 > 1$ such that
\[
\abs*{\eqref{expansion:4}}
\lesssim
\frac{1}{\eta^2}
\abs*{\sum_{1 \leq k \leq K} \delta_k \Phi \parens{\zeta_r^k}}^2
+
\frac{1}{\eta^p}
\abs*{\sum_{1 \leq k \leq K} \delta_k \Phi \parens{\zeta_r^k}}^p
\lesssim
\frac{1}{r^{N - 1} \eta^4}
\]
for every $\eta \geq \eta_0$  and $r \in R_\eta$.
\end{proof}

Now, we use the previous result to expand
$\overline{\sigma_\eta}$.

\begin{cor}
\label{cor:expansion}
There exists $\eta_0 > 1$ such that
\begin{multline*}
\abs*{\overline{\sigma_\eta} \parens{r} - F_\eta \parens{r}}
\lesssim
r^{\frac{\parens{3 - p} N + p - 1}{p '}}
\times
\begin{cases}
\eta^{- \min \parens*{
	4,
	\frac{6}{p '},
	2
	\parens*{
		2 \parens{p - 2} + \frac{1}{p '}
	}
}},
&\text{if} ~ K \in \set{2, 3};
\\
\\
\eta^{- \min \parens*{
	\ell,
	4,
	\frac{6}{p '},
	2
	\parens*{
		2 \parens{p - 2} + \frac{1}{p '}
	}
}},
&\text{if} ~ K \geq 4
\end{cases}
\end{multline*}
for every $\eta \geq \eta_0$ and
$r \in \overline{R_\eta}$.
\end{cor}
\begin{proof}
It follows from the Mean Value Inequality that
\[
\abs*{
	\overline{\sigma_\eta} \parens{r}
	-
	S_\eta \parens{W_{\eta, r}}
	-
	\eqref{cor:expansion:1}
}
\leq
\eqref{cor:expansion:2},
\]
where
\begin{equation}
\label{cor:expansion:1}
\angles*{\nabla S_\eta \parens{W_{\eta, r}} ~\middle|~ \nu_{\eta, r}}_{H^1_\eta}
\end{equation}
and
\begin{equation}
\label{cor:expansion:2}
\frac{1}{2}
\norm{\nu_{\eta, r}}_{H^1_\eta}^2
\parens*{
	\max_{t \in \cci{0, 1}}
		\norm*{S_\eta'' \parens{W_{\eta, r} + t \nu_{\eta, r}}}
}.
\end{equation}

\paragraph{Estimation of $\abs{\eqref{cor:expansion:1}}$.}
Due to Lemmas \ref{lem:pseudo-critical} and \ref{lem:implicit-function}, there exists
$\eta_0 > 1$ such that
\[
\abs*{\eqref{cor:expansion:1}}
\lesssim
\frac
	{r^{\frac{\parens{3 - p} N + p - 1}{p '}}}
	{
		\eta^{
			2
			\min
			\parens*{
				\frac{3}{p '},
				2 \parens{p - 2} + \frac{1}{p '}
			}
		}
	}
\]
for every $\eta \geq \eta_0$.

\paragraph{Estimation of $\abs{\eqref{cor:expansion:2}}$.}
It is easy to verify that
\[
\norm*{S_\eta'' \parens{W_{\eta, r}}}_{\mathrm{Bil} \parens{H^1_\eta}}
\leq
1 + \parens{p - 1} \norm{W_{\eta, r}}_{H^1_\eta}^{p - 2}.
\]
As such, we can use Lemmas \ref{lem:pseudo-critical}, \ref{lem:implicit-function} and \ref{lem:estimate-second-derivative} to estimate $\abs{\eqref{cor:expansion:2}}$ with the arguments in the previous paragraph.
\end{proof}

Finally, we can use the expansion of
$\overline{\sigma_\eta}$ to prove that
$\sigma_\eta$ has a minimum point and deduce the existence of solution to the bifurcation equation.

\begin{lem}
\label{lem:existence-of-minimum}
There exists $\eta_0 > 1$ such that given
$\eta \geq \eta_0$, $\sigma_\eta$ has a minimum point.
\end{lem}
\begin{proof}
The function $\overline{\sigma_\eta}$ is continuous and $\overline{R_{\eta}}$ is compact, so it is clear that $\overline{\sigma_\eta}$ has a minimum point. We want to prove that if $\eta$ is sufficiently large, then any minimum point of $\overline{\sigma_\eta}$ is not on $\partial R_\eta$. By contradiction, suppose that there exist
\begin{itemize}
\item
$\set{\eta_n}_{n \in \nat} \subset \ooi{1, \infty}$ such that $\eta_n \to \infty$ as
$n \to \infty$;
\item
a set
\[
\set*{
	r_{\eta_n} > 0:
	r_{\eta_n} \in \partial R_{\eta_n}
	\quad \text{and} \quad
	\overline{\sigma_{\eta_n}} \parens{r_{\eta_n}} = \min \overline{\sigma_{\eta_n}}
}_{n \in \nat}.
\]
\end{itemize}
To simplify the notation, we will henceforth denote the dependence on $\eta_n$ by the dependence on $n$. An application of the Pigeonhole Principle shows that, up to subsequence, one of the following holds:
\begin{enumerate}
\item
\label{case1}
$
\Phi \parens{r_n}^2 \parens{\int \Phi_{3 r_n} \Phi^{p - 1} \dif \Lebesgue}^{- 1}
=
\frac{\eta_n}{\log \eta_n}
$
for every $n \in \nat$ or
\item
\label{case2}
$
\Phi \parens{r_n}^2 \parens{\int \Phi_{3 r_n} \Phi^{p - 1} \dif \Lebesgue}^{- 1}
=
c \eta_n
$
for every $n \in \nat$.
\end{enumerate}
Before considering the cases, fix a collection
\[
\set*{
	\widetilde{r}_n > 0:
	\Phi \parens{\widetilde{r}_n^2}
	\parens*{\int \Phi_{3 \widetilde{r}_n} \Phi^{p - 1} \dif \Lebesgue}^{- 1}
	=
	\frac{c}{2} \eta_n
}_{n \in \nat}.
\]
Up to discarding a finite number of indices, $\widetilde{r}_n \in R_n$ for every
$n \in \nat$.

\paragraph{Case \eqref{case1}.}
Up to a higher order error estimated in Corollary \ref{cor:expansion}, it holds that
$F_n \parens{r_n} \leq F_n \parens{\widetilde{r}_n}$. That is,
\[
-
\frac{\Phi \parens{r_n}^2}{2 \eta}
+
\chi
\int \Phi_{3 r_n} \Phi^{p - 1} \dif \Lebesgue
\leq
-
\frac{\Phi \parens{\widetilde{r}_n}^2}{2 \eta}
+
\chi
\int \Phi_{3 \widetilde{r}_n} \Phi^{p - 1} \dif \Lebesgue.
\]
It follows from the definition of $\widetilde{r}_n$ and $r_n$ that
\[
\parens*{\chi - \frac{1}{2 \log \eta_n}}
\int \Phi_{3 r_n} \Phi^{p - 1} \dif \Lebesgue
\leq
\parens*{\chi - \frac{c}{4}}
\int \Phi_{3 \widetilde{r}_n} \Phi^{p - 1} \dif \Lebesgue.
\]
Due to Lemma \ref{lem:asymptotic},
\begin{equation}
\label{eqn:limit}
\frac
{
	\int
		\Phi_{3 \widetilde{r}_n} \Phi^{p - 1} 
	\dif \Lebesgue
}
{
	\int
		\Phi_{3 r_n} \Phi^{p - 1}	
	\dif \Lebesgue
}
\xrightarrow[n \to \infty]{}
1.
\end{equation}
We obtained a contradiction because $\chi > 0$ and $\chi - \frac{c}{4} < 0$.

\paragraph{Case \eqref{case2}.}
It suffices to argue as in the previous case to deduce that, up to a higher order error,
\[
\frac
{
	\int
		\Phi_{3 r_n} \Phi^{p - 1}	
	\dif \Lebesgue
}
{
	\int
		\Phi_{3 \widetilde{r}_n} \Phi^{p - 1} 
	\dif \Lebesgue1
}
\leq
\frac{1}{2} \times \frac{c - 4 \chi}{c - 2 \chi}
<
\frac{1}{2}
\]
for sufficiently large $n$. This fact is in contradiction with \eqref{eqn:limit}.
\end{proof}

\subsection{Proof of Theorem \ref{thm}}

Fix $\eta_0 > 1$ which satisfies the hypotheses of Lemmas \ref{lem:implicit-function}, \ref{lem:natural-constraint} and \ref{lem:existence-of-minimum}. Due to Lemma \ref{lem:existence-of-minimum}, there exists a collection
\[
\set*{
	r_\eta > 0:
	r_\eta \in R_\eta
	\quad \text{and} \quad
	\sigma_\eta \parens{r_\eta} = \min \sigma_\eta
}_{\eta \geq \eta_0}.
\]
Set
$u_\eta = W_{\eta, r_\eta} + \nu_{\eta, r_\eta}$ for every
$\eta \geq \eta_0$, so that Lemma \ref{lem:natural-constraint} ensures the equality
$\nabla S_\eta \parens{u_\eta} = 0$. To finish, the limits in the theorem follow from Lemmas \ref{lem:asymptotic}--\ref{lem:implicit-function}.

\qed

\appendix
\section{An equivalent problem}
\label{sect:equivalent-problem}

In this section, we proceed similarly as in \cite[Section 5]{fukayaStabilityInstabilityStanding2022} to show that, up to rescaling,
\begin{equation}
\label{eqn:delta-NLS-2}
- \Delta_\alpha u + \omega u = u \abs{u}^{p - 2}
\end{equation}
is equivalent to
\[- \Delta_{\alpha_\omega} u + u = u \abs{u}^{p - 2},\]
where
\begin{equation}
\label{eqn:alpha_omega}
\alpha_\omega
:=
\begin{cases}
{\displaystyle \alpha + \frac{1}{2 \pi} \log \sqrt{\omega}},
&\text{if} ~ N = 2;
\\
\\
{\displaystyle \frac{\alpha}{\sqrt{\omega}}},
&\text{if} ~ N = 3.
\end{cases}
\end{equation}
Suppose that
$u_\omega = \phi_\omega + q G_\omega$ solves \eqref{eqn:delta-NLS-2}. Consider the rescaling
\[
\widetilde{\phi}_\omega \parens{x}
:=
\omega^{- \frac{1}{p - 2}} \phi_\omega \parens{\omega^{- \frac{1}{2}} x},
\]
so that
\[
\phi_\omega \parens{x}
=
\omega^{\frac{1}{p - 2}}
\widetilde{\phi}_\omega \parens{\omega^{\frac{1}{2}} x}
\]
and
\[
- \Delta \phi_\omega \parens{x}
=
-
\omega^{\frac{p - 1}{p - 2}}
\Delta \widetilde{\phi}_\omega \parens{\omega^{\frac{1}{2}} x}
\]
Due to these equalities and \eqref{eqn:trick},
\[
\parens{- \Delta_\alpha + \omega} u_\omega \parens{x}
=
\parens{- \Delta + \omega} \phi_\omega \parens{x}
=
\omega^{\frac{p - 1}{p - 2}}
\parens{- \Delta + 1}
\widetilde{\phi}_\omega \parens{\omega^{\frac{1}{2}} x}
\]~
Similarly,
\[
u_\omega \parens{x} \abs*{u_\omega \parens{x}}^{p - 2}
=
\omega^{\frac{p - 1}{p - 2}}
\widetilde{u}_\omega \parens{x}
\abs*{\widetilde{u}_\omega \parens{x}}^{p - 2},
\]
where
\[
\widetilde{u}_\omega \parens{x}
:=
\omega^{- \frac{1}{p - 2}} u_\omega \parens{\omega^{- \frac{1}{2}} x}
=
\widetilde{\phi}_\omega \parens{x}
+
q
\omega^{\parens*{\frac{N - 2}{2} - \frac{1}{p - 2}}}
G \parens{x}.
\]
Therefore,
\[
- \Delta \widetilde{\phi}_\omega + \widetilde{\phi}_\omega
=
\widetilde{u}_\omega \abs{\widetilde{u}_\omega}^{p - 2}.
\]
Let us compute the quantity $\alpha_\omega \in \real$ such that
$\widetilde{u}_\omega \in \Dom \parens{- \Delta_{\alpha_\omega}}$.
On one hand,
\[
q \beta_\alpha \parens{\omega}
=
\phi_\omega \parens{0}
=
\omega^{\frac{1}{p - 2}}
\widetilde{\phi}_\omega \parens{0}.
\]
On the other hand,
\[
\widetilde{\phi}_\omega \parens{0}
=
q
\omega^{\parens*{\frac{N - 2}{2} - \frac{1}{p - 2}}}
\beta_{\alpha_\omega} \parens{1}.
\]
Therefore, either $q = 0$ or
$
\beta_{\alpha_\omega} \parens{1}
=
\omega^{- \frac{N - 2}{2}} \beta_\alpha \parens{\omega}
$.
This equality yields \eqref{eqn:alpha_omega} and we obtain
\[
- \Delta_{\alpha_\omega} \widetilde{u}_\omega + \widetilde{u}_\omega
=
\widetilde{u}_\omega \abs{\widetilde{u}_\omega}^{p - 2}.
\]

\section{Technical estimates}
\label{sect:technical}

For the sake of completeness, we provide the proof of the following elementary inequalities that are essential for many arguments in the paper.

\begin{lem}
\label{lem:elementary}
\begin{enumerate}
\item
\label{elementary:1}
If $0 < r \leq 1$, then
$\abs{\abs{a + b}^r - \abs{a}^r} \leq \abs{b}^r$
for every $a, b \in \real$.
\item
\label{elementary:2}
If $K \in \nat$ and $2 < r \leq 3$, then
\begin{multline*}
\abs*{
	\abs*{\sum_{1 \leq k \leq K} a_k}^r
	-
	\parens*{\sum_{1 \leq k \leq K} \abs{a_k}^r}
	-
	2
	\sum_{\substack{
		1 \leq k_1, k_2 \leq K;
		\\
		k_1 \neq k_2
	}}
	\parens*{a_{k_1} a_{k_2} \abs{a_{k_2}}^{r - 2}}
}
\leq
\\
\leq
r
\sum_{\substack{
	1 \leq k_1, k_2, k_3 \leq K;
	\\
	k_1 \neq k_3; k_2 \neq k_3
}}
	\parens*{\abs{a_{k_1}} \abs{a_{k_2}} \abs{a_{k_3}}^{r - 2}}
\end{multline*}
for every $a_1, \ldots, a_k \in \real$.
\item
\label{elementary:4}
If $M > 0$ and $r \geq 2$, then
\[
	\abs*{\abs{a + b}^r - \abs{a}^r - r b a \abs{a}^{r - 2}}
	\lesssim
	\abs{b}^2 + \abs{b}^r
\]
for every $a \in \cci{- M, M}$ and $b \in \real$.
\end{enumerate}
\end{lem}
\begin{proof}
~\paragraph{Proof of \eqref{elementary:1}.}
The function
$\coi{0, \infty} \ni t \mapsto t^r$
is increasing and sub-additive, so
\[\abs{a + b}^r \leq \parens*{\abs{a} + \abs{b}}^r \leq \abs{a}^r + \abs{b}^r\]
and
\[
\abs{a}^r
=
\abs*{\parens{a + b} + \parens{- b}}^r
\leq
\parens*{\abs{a + b} + \abs{b}}^r
\leq
\abs{a + b}^r + \abs{b}^r,
\]
hence the result.

\paragraph{Proof of \eqref{elementary:2}.} Consider the function
$g \colon \real^K \to \real$ defined as
\[
g \parens{x_1, \ldots, x_K}
=
\abs*{\sum_{1 \leq k \leq K} x_k}^r
-
\parens*{\sum_{1 \leq k \leq K} \abs{x_k}^r}
-
2
\sum_{\substack{1 \leq k_1, k_2 \leq K; \\ k_1 \neq k_2}}
	\parens*{x_{k_1} x_{k_2} \abs{x_{k_2}}^{r - 2}}.
\]
It is easy to verify that
\begin{multline*}
\partial_{k_3} g \parens{a_1, \ldots, a_K}
=
r
\parens*{\sum_{1 \leq k_2 \leq K} a_{k_2}}
\abs*{\sum_{1 \leq k_2 \leq K} a_{k_2}}^{r - 2}
-
r a_{k_3} \abs{a_{k_3}}^{r - 2}
-
\\
-
2
\sum_{\substack{1 \leq k_2 \leq K; \\ k_2 \neq k_3}}
	\parens*{a_{k_2} \abs{a_{k_2}}^{r - 2}}
-
2 \parens{r - 1}
\sum_{\substack{
	1 \leq k_1 \leq K;
	\\
	k_1 \neq k_3
}}
	\parens*{a_{k_1} \abs{a_{k_3}}^{r - 2}}.
\end{multline*}
Due to the Mean Value Theorem, there exists $c \in \ooi{0, 1}$ such that
\[
g \parens{a_1, \ldots, a_K}
=
c^{r - 1}
\sum_{1 \leq k \leq K} \parens*{\partial_k g \parens{a_1, \ldots, a_K} a_k}.
\]
Therefore,
\[
\abs*{g \parens{a_1, \ldots, a_K}}
\leq
\abs*{\sum_{1 \leq k \leq K} \parens*{\partial_k g \parens{a_1, \ldots, a_K} a_k}}.
\]
Let us estimate the RHS on this inequality. It is easy to verify that
\begin{multline*}
\sum_{1 \leq k_1 \leq K}
	\parens*{\partial_{k_1} g \parens{a_1, \ldots, a_K} a_{k_1}}
=
r \abs*{\sum_{1 \leq k_1 \leq K} a_{k_1}}^r
-
r \parens*{\sum_{1 \leq k_1 \leq K} \abs{a_{k_1}}^r}
-
\\
-
2
\sum_{\substack{1 \leq k_1, k_2 \leq K; \\ k_1 \neq k_2}}
	\parens*{a_{k_1} a_{k_2} \abs{a_{k_2}}^{r - 2}}
-
2 \parens{r - 1}
\sum_{\substack{1 \leq k_1, k_2 \leq K; \\ k_1 \neq k_2}}
	\parens*{a_{k_1} a_{k_2} \abs{a_{k_2}}^{r - 2}}
=
\\
=
r  \abs*{\sum_{1 \leq k_1 \leq K} a_{k_1}}^r
-
r \parens*{\sum_{1 \leq k_1 \leq K} \abs{a_{k_1}}^r}
-
2 r
\sum_{\substack{1 \leq k_1, k_2 \leq K; \\ k_1 \neq k_2}}
	\parens*{a_{k_1} a_{k_2} \abs{a_{k_2}}^{r - 2}}.
\end{multline*}
Clearly,
\begin{multline*}
\abs*{\sum_{1 \leq k \leq K} a_k}^r
=
\parens*{\sum_{1 \leq k \leq K} a_k}^2
\abs*{\sum_{1 \leq k \leq K} a_k}^{r - 2}
=
\\
=
\sum_{1 \leq k_1 \leq K}
	\parens*{a_{k_1}^2 \abs*{\sum_{1 \leq k_2 \leq K} a_{k_2}}^{r - 2}}
+
\sum_{\substack{1 \leq k_1, k_2 \leq K; \\ k_1 \neq k_2}}
	\parens*{a_{k_1} a_{k_2} \abs*{\sum_{1 \leq k_3 \leq K} a_{k_3}}^{r - 2}}.
\end{multline*}
By using this identity, we obtain
\begin{multline*}
\sum_{1 \leq k_1 \leq K}
	\parens*{\partial_{k_1} g \parens{a_1, \ldots, a_K} a_{k_1}}
=
\\
=
r
\sum_{1 \leq k_1 \leq K}
	\parens*{
		a_{k_1}^2
		\parens*{\abs*{\sum_{1 \leq k_2 \leq K} a_{k_2}}^{r - 2} - \abs{a_{k_1}}^{r - 2}}
	}
+
\\
+
r
\sum_{\substack{1 \leq k_1, k_2 \leq K; \\ k_1 \neq k_2}}
	\parens*{
		a_{k_1} a_{k_2}
		\parens*{
			\abs*{\sum_{1 \leq k_3 \leq K} a_{k_3}}^{r - 2}
			-
			\abs{a_{k_1}}^{r - 2}
			-
			\abs{a_{k_2}}^{r - 2}
		}
	}.
\end{multline*}
At this point, the result follows from \eqref{elementary:1}.

\paragraph{Proof of \eqref{elementary:4}.} By the Mean Value Inequality,
\begin{align*}
\abs*{\abs{a + b}^r - \abs{a}^r - r b a \abs{a}^{r - 2}}
&\leq
\frac{r \parens{r - 1}}{2}
\parens*{\max_{t \in \cci{0, 1}} \abs{a + t b}^{r - 2}}
\abs{b}^2;
\\
&\leq
\frac{r \parens{r - 1} 2^{r - 2}}{2}
\parens*{M^{r - 2} \abs{b}^2 + \abs{b}^r},
\end{align*}
hence the result.
\end{proof}

Our next results are concerned with the estimation of the effect of the power nonlinearity on functions with exponential decay. We begin with an estimation on exterior domains.

\begin{lem}
\label{lem:outside-the-ball}
Suppose that $r > 1$; $\alpha_1 \geq 0$;
$\alpha_2 > 0$ and
$\parens{r - 1} \alpha_2 > \alpha_1$. Suppose further that $N \in \nat$;
$0 \leq \beta_1 < N$;
$\beta_2 \geq 0$; $0 \leq \eta, \delta \leq 1$ and $u_1, u_2$ are functions defined $\Lebesgue$-a.e. in $\real^N$ such that
\begin{equation}
\label{eqn:L^infty}
y \mapsto e^{\alpha_1 \abs{y}} \abs{y}^{\beta_1} u_1 \parens{y}
\quad \text{and} \quad
y \mapsto e^{\alpha_2 \abs{y}} \abs{y}^{\beta_2} u_2 \parens{y}
\quad \text{are in} \quad L^\infty.
\end{equation}
Then there exists $M > 0$ such that
\begin{multline*}
\int_{\real^N \setminus B \parens*{0, \eta \abs{y}^\delta}}
	\abs*{u_1 \parens{x + y}} \abs*{u_2 \parens{x}}^{r - 1}
\dif \Lebesgue \parens{x}
\lesssim
\abs{y}^{\delta \parens*{N - \beta_1 - \parens{r - 1} \beta_2}}
e^{- \alpha_1 \abs{y}}
\times
\\
\times
\parens*{
	e^{
		- 
		\parens*{\parens{r - 1} \alpha_2 - \alpha_1}
		\parens*{\abs{y} - \eta \abs{y}^\delta}
	}
	+
	e^{
		- 
		\parens*{\parens{r - 1} \alpha_2 - \alpha_1}
		\eta \abs{y}^\delta
	}
}
\end{multline*}
for every $y \in \real^N \setminus B \parens{0, M}$.
\end{lem}
\begin{proof}
Clearly,
\begin{multline*}
\int_{\real^N \setminus B \parens*{0, \eta \abs{y}^\delta}}
	\abs*{u_1 \parens{x + y}} \abs*{u_2 \parens{x}}^{r - 1}
\dif \Lebesgue \parens{x}
=
\\
=
\int_{\real^N \setminus B \parens*{0, \eta \abs{y}^\delta}}
	\frac{
		e^{- \alpha_1 \abs{x + y}}
		e^{- \parens{r - 1} \alpha_2 \abs{x}}
	}{
		\abs{x + y}^{\beta_1}
		\abs{x}^{\parens{r - 1} \beta_2}
	}
	\times
\\
	\times
	\parens*{e^{\alpha_1 \abs{x + y}} \abs{x + y}^{\beta_1} \abs*{u_1 \parens{x + y}}}
	\parens*{e^{\alpha_2 \abs{x}} \abs{x}^{\beta_2} \abs*{u_2 \parens{x}}}^{r - 1}
\dif \Lebesgue \parens{x}.
\end{multline*}
In view of \eqref{eqn:L^infty}, we only have to estimate
\begin{equation}
\label{outside-the-ball:1}
\int_{\real^N \setminus B \parens*{0, \eta \abs{y}^\delta}}
	\frac{
		e^{- \alpha_1 \abs{x + y}}
		e^{- \parens{r - 1} \alpha_2 \abs{x}}
	}{
		\abs{x + y}^{\beta_1}
		\abs{x}^{\parens{r - 1} \beta_2}
	}
\dif \Lebesgue \parens{x}.
\end{equation}
Due to the considered domain of integration and the Triangle Inequality,
\[
\eqref{outside-the-ball:1}
\leq
\frac{1}{\eta^{\parens{r - 1} \beta_2}}
\times
\frac{1}{e^{\alpha_1 \abs{y}} \abs{y}^{\parens{r - 1} \beta_2 \delta}}
\int_{\real^N \setminus B \parens*{0, \eta \abs{y}^\delta}}
	\frac{
		e^{- \parens*{\parens{r - 1} \alpha_2 - \alpha_1} \abs{x}}
	}{\abs{x + y}^{\beta_1}}
\dif \Lebesgue \parens{x}.
\]
There exists $M > 0$ such that
\begin{multline*}
\frac{1}{
	e^{\alpha_1 \abs{y}}
	\abs{y}^{\parens{r - 1} \beta_2 \delta}
}
\int_{B \parens*{- y, \eta \abs{y}^\delta}}
	\frac{
		e^{- \parens*{\parens{r - 1} \alpha_2 - \alpha_1} \abs{x}}
	}{\abs{x + y}^{\beta_1}}
\dif \Lebesgue \parens{x}
=
\\
=
\frac{1}{
	e^{\alpha_1 \abs{y}}
	\abs{y}^{\parens{r - 1} \beta_2 \delta}
}
\int_{B \parens*{0, \eta \abs{y}^\delta}}
	\frac{
		e^{- \parens*{\parens{r - 1} \alpha_2 - \alpha_1} \abs{z - y}}
	}{\abs{z}^{\beta_1}}
\dif \Lebesgue \parens{z}
\lesssim
\\
\lesssim
\abs{y}^{\delta \parens*{N - \beta_1 - \parens{r - 1} \beta_2}}
e^{- \alpha_1 \abs{y}}
e^{
	- 
	\parens*{\parens{r - 1} \alpha_2 - \alpha_1}
	\parens*{\abs{y} - \eta \abs{y}^\delta}
}
\end{multline*}
and
\begin{multline*}
\frac{1}{
	e^{\alpha_1 \abs{y}}
	\abs{y}^{\parens{r - 1} \beta_2 \delta}
}
\int_{
	\real^N
	\setminus
	\parens*{
		B \parens*{0, \eta \abs{y}^\delta} \cup B \parens*{- y, \eta \abs{y}^\delta}
	}
}
	\frac{
		e^{- \parens*{\parens{r - 1} \alpha_2 - \alpha_1} \abs{x}}
	}{\abs{x + y}^{\beta_1}}
\dif \Lebesgue \parens{x}
\lesssim
\\
\lesssim
\abs{y}^{
	\delta
	\parens*{
		N - 1 - \beta_1 - \parens{r - 1} \beta_2
	}
}
e^{- \alpha_1 \abs{y}}
e^{
	- 
	\parens*{\parens{r - 1} \alpha_2 - \alpha_1}
	\eta \abs{y}^\delta
}
\end{multline*}
for every
$y \in \real^N \setminus B \parens{0, M}$, hence the result.
\end{proof}

Now, we estimate on bounded domains.

\begin{lem}
\label{lem:inside-the-ball}
Suppose that $r \geq 0$;
$\alpha_1, \alpha_2 \geq 0$ are such that
$\parens{r - 1} \alpha_2 \geq \alpha_1$ and
$0 \leq \eta, \delta \leq 1$ are such that
$\eta \delta > 0$. Suppose further that
$N \in \nat$; $\beta_1 \geq 0$;
$0 \leq \beta_2 < \frac{N}{r - 1}$ and
$u_1, u_2$ are functions defined $\Lebesgue$-a.e. in $\real^N$ such that \eqref{eqn:L^infty} is satisfied. Then there exists
$M > 0$ such that
\[
\int_{B \parens*{0, \eta \abs{y}^\delta}}
	\abs*{u_1 \parens{x + y}} \abs*{u_2 \parens{x}}^{r - 1}
\dif \Lebesgue \parens{x}
\lesssim
\abs{y}^{\delta \parens*{N - \parens{r - 1} \beta_2} - \beta_1}
e^{-\alpha_1 \abs{y}}
\]
for every $y \in \real^N \setminus B \parens{0, M}$.
\end{lem}
\begin{proof}
Similarly as in the previous proof, we only have to estimate
\begin{equation}
\label{inside-the-ball:1}
\int_{B \parens*{0, \eta \abs{y}^\delta}}
	\frac{
		e^{- \alpha_1 \abs{x + y}}
		e^{- \parens{r - 1} \alpha_2 \abs{x}}
	}{
		\abs{x + y}^{\beta_1}
		\abs{x}^{\parens{r - 1} \beta_2}
	}
\dif \Lebesgue \parens{x}.
\end{equation}
Due to the Triangle Inequality, there exists $M \in \ooi{0, \infty}$ such that
\begin{multline*}
\eqref{inside-the-ball:1}
\lesssim
\frac{1}{
	\parens*{\abs{y} - \eta \abs{y}^\delta}^{\beta_1} e^{\alpha_1 \abs{y}}
}
\int_{B \parens*{0, \eta \abs{y}^\delta}}
	\frac{e^{-\parens*{\parens{r - 1} \alpha_2 - \alpha_1} \abs{x}}}{
		\abs{x}^{\parens{r - 1} \beta_2}
	}
\dif \Lebesgue \parens{x}
\lesssim
\\
\lesssim
\abs{y}^{\delta \parens*{N - \parens{r - 1} \beta_2} - \beta_1}
e^{-\alpha_1 \abs{y}}
\end{multline*}
for every $y \in \real^N \setminus B \parens{0, M}$.
\end{proof}

Finally, we obtain an estimate on the entire space $\real^N$ by considering the case
$\delta = 1$ and $\eta = \frac{1}{2}$ in the previous lemmas.

\begin{cor}
\label{cor:entire-space}
Suppose that $r > 1$; $\alpha_1 \geq 0$;
$\alpha_2 > 0$ and
$\parens{r - 1} \alpha_2 > \alpha_1$. Suppose further that $N \in \nat$,
$0 \leq \beta_1 < N$;
$0 \leq \beta_2 < \frac{N}{r - 1}$ and $u_1, u_2$ are functions defined $\Lebesgue$-a.e. in
$\real^N$ such that \eqref{eqn:L^infty} is satisfied. Then there exists
$M > 0$ such that
\[
\int
	\abs*{u \parens{x + y}} \abs*{u_2 \parens{x}}^{r - 1}
\dif \Lebesgue \parens{x}
\lesssim	
\abs{y}^{N - \beta_1 - \parens{r - 1} \beta_2}
e^{- \alpha_1 \abs{y}}
\]
for every $y \in \real^N \setminus B \parens{0, M}$.
\end{cor}

\section{Asymptotic behavior of $\Phi$}
\label{sect:asymptotic}

In this appendix, we develop a relatively self-contained estimate of the error associated with the limit \eqref{eqn:Phi-limit}. That is, we want to estimate
\[\abs*{\abs{y}^{\frac{N - 1}{2}} e^{\abs{y}} \Phi \parens{y} - \theta_\Phi}\]
when $y$ is sufficiently away from the origin in $\real^N$. We begin with an important collection of estimates.

\begin{lem}
\label{lem:angle}
If $0 \leq \delta < \frac{1}{2}$, then there exists $M_\delta > 0 $ such that
\[
\abs*{
	\abs{y + x} - \abs{y} - \frac{y \cdot x}{\abs{y}}
}
\lesssim
\frac{1}{\abs{y}^{1 - 2 \delta}},
\quad \quad
\abs*{
	\frac{1}{\abs{y + x}^{\frac{N - 1}{2}}}
	-
	\frac{1}{\abs{y}^{\frac{N - 1}{2}}}
}
\lesssim
\frac{1}{\abs{y}^{\frac{N + 1 - 2 \delta}{2}}},
\]
\[
\abs*{e^{- \abs{y + x}} - e^{- \parens*{\abs{y} + \frac{y \cdot x}{\abs{y}}}}}
\lesssim
\frac{e^{- \parens*{\abs{y} + \frac{y \cdot x}{\abs{y}}}}}{\abs{y}^{1 - 2 \delta}}
\]
and
\[
\abs*{
	\frac{e^{- \abs{y + x}}}{\abs{y + x}^{\frac{N - 1}{2}}}
	-
	\frac{e^{- \parens*{\abs{y} + \frac{y \cdot x}{\abs{y}}}}}{\abs{y}^{\frac{N - 1}{2}}}
}
\lesssim
\frac{e^{- \abs{y + x}}}{\abs{y}^{\frac{N + 1 - 2 \delta}{2}}}
+
\frac{e^{- \parens*{\abs{y} + \frac{y \cdot x}{\abs{y}}}}}
	{\abs{y}^{\frac{N + 1 - 4 \delta}{2}}}
\]
for every $y \in \real^N \setminus B \parens{0, M_\delta}$
and $x \in B \parens{0, \abs{y}^\delta}$.
\end{lem}
\begin{proof}
~\paragraph{Proof of the first estimate.}
A difference of squares shows that
\[
\abs{y + x}
-
\abs{y}
-
\frac{y \cdot x}{\abs{y}}
=
\frac{
	\parens{y \cdot x} \parens*{\abs{y} - \abs{y + x}}
	+
	\abs{x}^2 \abs{y}
}
{\parens*{\abs{y} + \abs{y + x}} \abs{y}}.
\]
Due to the Triangle Inequality,
\[
\abs*{\abs{y} - \abs{x - y} - \frac{x \cdot y}{\abs{y}}}
\leq
\frac{2 \abs{x}^2}{\abs{x + y} + \abs{y}},
\]
hence the result.

\paragraph{Proof of the second estimate.}
Follows from differences of squares as in the previous proof.

\paragraph{Proof of the third estimate.}
Clearly,
\[
\abs*{e^{- \abs{y + x}} - e^{- \parens*{\abs{y} + \frac{y \cdot x}{\abs{y}}}}}
=
e^{- \parens*{\abs{y} + \frac{y \cdot x}{\abs{y}}}}
\abs*{
	e^{\parens*{\abs{y} + \frac{y \cdot x}{\abs{y}} - \abs{y + x}}}
	-
	1
}.
\]
As such, the result follows from the first estimate.

\paragraph{Proof of the fourth estimate.} Corollary of the second and third estimates.
\end{proof}

The next result is often implicitly used in texts that employ Lyapunov--Schmidt reduction.	

\begin{lem}
\label{lem:symmetric-integral}
Let $N \in \nat$ and $v$ be a radial function defined $\Lebesgue$-a.e. in $\real^N$ such that either $v \geq 0$ $\Lebesgue$-a.e. or
$\int e^{- x_1} \abs{v \parens{x}} \dif \Lebesgue \parens{x} < \infty$.
Then
\[
\int
	e^{- \frac{x \cdot y}{\abs{y}}} v \parens{x}
\dif \Lebesgue \parens{x}
=
\int
	e^{- x \cdot z} v \parens{x}
\dif \Lebesgue \parens{x}
\]
for every $y \in \real^N \setminus \set{0}$ and $z \in \sphere^{N - 1}$.
\end{lem}
\begin{proof}
For the sake of simplicity, suppose that $N = 2$. Let
$x' = \abs{y}^{- 1} y \in \sphere^1$ and let $y' \in \sphere^1$ be such that
$\parens{x', y'}$ is a positively-oriented orthonormal basis of $\real^2$. The application
\[
\real^2 \ni z
\mapsto
\parens{z \cdot x', z \cdot y'} \in \real \times \real \simeq \real^2
\]
is an isometry, so a change of variable shows that
\[
\int e^{- z \cdot x'} v \parens{z} \dif \Lebesgue \parens{z}
=
\int e^{- z_1} v \parens{z} \dif \Lebesgue \parens{z},
\]
hence the result.
\end{proof}

The following expansion of the Green's function $G$ is a corollary of the expansion in \cite[Section 9.7]{abramowitzHandbookMathematicalFunctions1972}.
\begin{lem}
\label{lem:expansion-of-Greens}
In the case $N = 2$, there exists $M > 0$ such that
\[
\abs*{
	G \parens{z}
	-
	\frac{e^{- \abs{z}}}
		{
			2^{\frac{3}{2}}
			\pi^{\frac{1}{2}}
			\abs{z}^{\frac{1}{2}}
		}
}
\lesssim
\frac{e^{- \abs{z}}}{\abs{z}^{\frac{3}{2}}}
\]
for every
$z \in \real^2 \setminus B \parens{0, M}$.
\end{lem}

Recall that $\Phi = \Phi^{p - 1} \ast G$ (see \cite[p. 386]{gidasSymmetryPositiveSolutions1981}). The next result is related to this identity and the previous expansion of $G$.

\begin{lem}
\label{lem:error-in-limit:bounded-domain}
If $0 \leq \delta < \frac{1}{2}$, then there exists $M_\delta > 0$ such that
\[
\abs*{
	\int_{B \parens*{0, \abs{y}^\delta}}
		\frac{e^{- \abs{y - x}}}{\abs{y - x}^{\frac{N - 1}{2}}}
		\Phi \parens{x}^{p - 1}
	\dif \Lebesgue \parens{x}
	-
	\theta_\Phi
	\frac{e^{- \abs{y}}}{\abs{y}^{\frac{N - 1}{2}}}
}
\lesssim
\frac
	{
		\abs{y}^{
			\frac
				{\delta \parens*{
					p - 1 + \parens{3 - p} N}
				}
				{2}
		}
	}
	{\abs{y}^{\frac{N + 1 - 4 \delta}{2}}}
e^{- \abs{y}}
\]
for every $y \in \real^N \setminus B \parens{0, M_\delta}$.
\end{lem}
\begin{proof}
Due to Lemmas \ref{lem:inside-the-ball} and \ref{lem:angle}, there exists
$M_\delta > 0$ such that
\begin{multline*}
\abs*{
	\int_{B \parens*{0, \abs{y}^\delta}}
		\parens*{
			\frac{e^{- \abs{y - x}}}{\abs{y - x}^{\frac{N - 1}{2}}}
			-
			\frac{e^{- \parens*{\abs{y} - \frac{y \cdot x}{\abs{y}}}}}
				{\abs{y}^{\frac{N - 1}{2}}}
		}
		\Phi \parens{x}^{p - 1}
	\dif \Lebesgue \parens{x}
}
\lesssim
\\
\lesssim
\frac{1}{\abs{y}^{\frac{N + 1 - 2 \delta}{2}}}
\int_{B \parens*{0, \abs{y}^\delta}}
		e^{- \abs{y - x}}
		\Phi \parens{x}^{p - 1}
	\dif \Lebesgue \parens{x}
+
\\
+
\frac{e^{- \abs{y}}}{\abs{y}^{\frac{N + 1 - 4 \delta}{2}}}
\int_{B \parens*{0, \abs{y}^\delta}}
	e^{\frac{y \cdot x}{\abs{y}}}
	\Phi \parens{x}^{p - 1}
\dif \Lebesgue \parens{x}
\lesssim
\frac
	{
		\abs{y}^{
			\frac
				{\delta \parens*{
					p - 1 + \parens{3 - p} N}
				}
				{2}
		}
	}
	{\abs{y}^{\frac{N + 1 - 4 \delta}{2}}}
e^{- \abs{y}}
\end{multline*}
for every $y \in \real^N \setminus B \parens{0, M_\delta}$. At this point, the result follows from Lemma \ref{lem:symmetric-integral} and the definition of the positive constant
$\theta_\Phi$.
\end{proof}

We can finally obtain an estimate of
$\abs{\abs{z}^{\frac{N - 1}{2}} e^{\abs{z}} \Phi \parens{z} - \theta_\Phi}$
for sufficiently large $\abs{z}$.

\begin{lem}
\label{lem:error-in-limit}
If $0 < \eps < 1$, then there exists
$M_\eps > 0$ such that
\[
\abs*{
	\Phi \parens{y}
	-
	\theta_\Phi
	\frac{e^{- \abs{y}}}{\abs{y}^{\frac{N - 1}{2}}}
}
\lesssim
\frac{e^{- \abs{y}}}{\abs{y}^{\frac{N + 1 - \eps}{2}}}
\]
for every $y \in \real^N \setminus B \parens{0, M_\eps}$.
\end{lem}
\begin{proof}
Fix $\delta \in \ooi{0, \frac{1}{2}}$.

\paragraph{The case $N = 3$.}
As $K_{\frac{1}{2}} \parens{t} = \sqrt{\frac{\pi}{2 t}} e^{- t}$, we only have to estimate
\begin{equation}
\label{lem:error-in-limit:1}
\int_{\real^3 \setminus B \parens*{0, \abs{y}^\delta}}
	\frac{e^{- \abs{y - x}}}{\abs{y - x}} \Phi \parens{x}^{p - 1}
\dif \Lebesgue \parens{x}
\end{equation}
in function of $\abs{y}$ to conclude from an application of Lemma \ref{lem:error-in-limit:bounded-domain}. It follows from Lemma \ref{lem:outside-the-ball} that there exists
$M > 0$ such that
\[
\eqref{lem:error-in-limit:1}
\lesssim
\abs{y}^{2 \delta}
e^{- \abs{y}} e^{- \parens{p - 2} \abs{y}^\delta}
\]
for every $y \in \real^3 \setminus B \parens{0, M}$, hence the result.

\paragraph{The case $N = 2$.}

\subparagraph{Estimation on exterior domain.}
Due to Lemma \ref{lem:outside-the-ball}, there exists $M_\delta > 0$ such that
\[
\abs*{
	\int_{\real^2 \setminus B \parens{0, \abs{y}^\delta}}
		G \parens{x - y}
		\Phi \parens{x}^{p - 1}
	\dif \Lebesgue \parens{x}
}
\lesssim
\abs{y}^{\frac{\delta \parens{4 - p}}{2}}
e^{- \abs{y}}
e^{- \parens{p - 2} \abs{y}^\delta}
\]
for every $y \in \real^2 \setminus B \parens{0, M_\delta}$ and
$x \in B \parens{0, \abs{y}^\delta}$.

\subparagraph{Expansion on bounded domain.}
In view of Lemmas \ref{lem:inside-the-ball} and \ref{lem:expansion-of-Greens}, we can associate any $\eta \in \ooi{0, 1}$ with an
$M_{\delta, \eta} > 0$ such that
\begin{multline*}
\abs*{
	\int_{B \parens{0, \abs{y}^\delta}}
		\parens*{
			G \parens{x - y}
			-
			\frac
				{e^{- \abs{x - y}}}
				{
					2^{\frac{3}{2}}
					\pi^{\frac{1}{2}}
					\abs{x - y}^{\frac{1}{2}}
				}
		}
		\Phi \parens{x}^{p - 1}
	\dif \Lebesgue \parens{x}
}
\lesssim
\\
\lesssim
\int_{B \parens{0, \abs{y}^\delta}}
	\frac{e^{- \abs{x - y}}}{\abs{x - y}^{\frac{3}{2}}}
	\Phi \parens{x}^{p - 1}
\dif \Lebesgue \parens{x}
\lesssim
\\
\lesssim
\frac{1}{
	\parens*{
		\abs{y} - \abs{y}^\delta
	}^{\frac{3}{2}}
}
\int_{B \parens{0, \abs{y}^\delta}}
	e^{- \abs{x - y}}
	\Phi \parens{x}^{p - 1}
\dif \Lebesgue \parens{x}
\lesssim
\frac
	{\abs{y}^{\frac{\delta \parens{5 - p}}{2}}}
	{\abs{y}^{\frac{3}{2}}}
e^{- \abs{y}}
\end{multline*}
for every $y \in \real^2 \setminus B \parens{0, M_{\delta, \eta}}$ and
$x \in B \parens{0, \abs{y}^\delta}$. At this point, the result follows from an application of Lemma \ref{lem:error-in-limit:bounded-domain}.
\end{proof}

We seize the opportunity to obtain an approximation of
\[
\int
	\Phi \parens{y + x} \Phi \parens{x}^{p - 1}
\dif \Lebesgue \parens{x}
\]
for sufficiently large $\abs{y}$.

\begin{lem}
\label{lem:interaction}
If $0 < \eps < 1$, then there exists $M_\eps > 0$ such that
\[
\abs*{
	\int
		\Phi \parens{x + y}
		\Phi \parens{x}^{p - 1}
	\dif \Lebesgue \parens{x}
	-
	\theta_\Phi
	\Phi \parens{y}
}
\lesssim	
\frac{e^{- \abs{y}}}{\abs{y}^{\frac{N + 1 - \eps}{2}}},
\]
for every $y \in \real^N \setminus B \parens{0, M_\eps}$.
\end{lem}
\begin{proof}
Fix $\delta \in \ooi{0, \frac{1}{2}}$.

\paragraph{Estimation on exterior domain.}
Due to Lemma \ref{lem:outside-the-ball}, there exists $M \in \ooi{0, \infty}$ such that
\[
\int_{\real^N \setminus B \parens*{0, \abs{y}^\delta}}
	\Phi \parens{x + y} \Phi \parens{x}^{p - 1}
\dif \Lebesgue \parens{x}
\lesssim
\abs{y}^{\frac{\delta p}{2}}
\abs{y}^{\frac{N + 1}{2}} e^{- \abs{y}} e^{- \parens{p - 2} \abs{y}}
\]
for every $y \in \real^N \setminus B \parens{0, M}$.

\paragraph{Expansion on bounded domain.}
In view of Lemmas \ref{lem:inside-the-ball} and \ref{lem:error-in-limit}, we can associate any $\eta \in \ooi{0, 1}$ with an $M_\eta > 0$ such that
\begin{multline*}
\abs*{
	\int_{B \parens*{0, \abs{y}^\delta}}
		\parens*{
			\Phi \parens{x + y}
			-
			\theta_\Phi \frac{e^{- \abs{x + y}}}{\abs{x + y}^{\frac{N - 1}{2}}}
		}
		\Phi \parens{x}^{p - 1}
	\dif \Lebesgue \parens{x}
}
\lesssim
\\
\lesssim
\int_{B \parens*{0, \abs{y}^\delta}}
	\frac{e^{- \abs{x + y}}}{\abs{x + y}^{\frac{N + 1 - \eps}{2}}}
	\Phi \parens{x}^{p - 1}
\dif \Lebesgue \parens{x}
\lesssim
\frac{e^{- \abs{y}}}{
	\abs{y}^{\parens*{
		\frac{N + 1 - \eps}{2} - \delta \frac{\parens{p - 1} \parens{N - 1}}{2} - \eta
	}}
}
\end{multline*}
for every $y \in \real^N \setminus B \parens{0, M_\eta}$. At this point, the result follows from Lemma \ref{lem:error-in-limit:bounded-domain}.
\end{proof}

\sloppy
\printbibliography
\end{document}